\documentclass[12pt]{article}

\pdfoutput=1

\usepackage{amsmath}

\usepackage{latexsym} 

\usepackage{amssymb,mathrsfs} 
\usepackage{mathrsfs}

\usepackage{amsthm} 

\usepackage[all]{xy}
\CompileMatrices

\usepackage{eso-pic}
\usepackage{graphicx,color}
\usepackage{type1cm}

\definecolor{red}{rgb}{1.,0.,0.}
\definecolor{darkred}{cmyk}{0.,1.,1.,.4}
\definecolor{cyan}{cmyk}{1.,0.,0.,.5}
\definecolor{green}{cmyk}{1.,0.,1.,.5}
\definecolor{lightgreen}{cmyk}{.5,0.,.5,.0}
\definecolor{blue}{cmyk}{1.,1.,0.,0.}
\definecolor{darkblue}{cmyk}{1.,1.,0.,.5}
\definecolor{lightblue}{cmyk}{.5,.5,0.,0.}
\definecolor{LIGHTBLUE}{cmyk}{.5,.5,0.,0.}
\definecolor{paleblue}{cmyk}{.2,.2,0.,0.}
\definecolor{paleyellow}{cmyk}{0.,0.,.2,0.}
\definecolor{red}{cmyk}{1.,0.,1.,0.}
\definecolor{faded}{gray}{.75}
\definecolor{orange}{rgb}{1.,.30,0.}
\definecolor{purple}{rgb}{.48,0.,.48}

\newcommand{\mc}{\mathcal}
\newcommand{\mb}{\mathbb} 
\newcommand{\mf}{\mathfrak}
\newcommand{\ms}{\mathscr} 		
\newcommand{\msans}{\mathsf} 

\newcommand{\ds}{\displaystyle}

\newcommand{\Hom}{\operatorname{Hom}}

\newcommand{\mto}{\leadsto}

\theoremstyle{plain}
\newtheorem{theorem}{Theorem}

\newtheorem{corollary}[theorem]{Corollary}
\newtheorem{lemma}[theorem]{Lemma}
\newtheorem{proposition}[theorem]{Proposition}

\newtheorem{convention}{Convention}

\theoremstyle{definition}
\newtheorem{definition}[theorem]{Definition}

\newtheorem{notation}{Notation}

\theoremstyle{remark}
\newtheorem*{remark}{Remark}
\newtheorem*{observation}{Observation}

\theoremstyle{plain}

\newcommand{\bounded}{L^{\infty}_{\text{loc}}}
\newcommand{\continuous}{C^0}
\newcommand{\smooth}{C^{\infty}}

\newcommand{\pluriharmonic}{\sheaf{H}}
\newcommand{\sheaf}{\ms}

\newcommand{\ed}{\msans{d}} 

\newcommand{\abs}[1]{\vert #1\vert}

\newcommand{\logabs}[1]{\log\vert #1\vert}

\newcommand{\rest}[1]{{\big\arrowvert_{#1}}}

\DeclareMathOperator{\supp}{Supp}

\newcommand{\id}[1]{{\operatorname{id}_{#1}}}

\newcommand{\N}{\mathbb{N}}
\newcommand{\R}{\mathbb{R}}
\newcommand{\C}{\mathbb{C}}
\newcommand{\Z}{\mathbb{Z}}

\newcommand{\Proj}{\mathbb{P}}

\newcommand{\cp}[1]{^{\circ{#1}}}

\newcommand{\topcat}{\operatorname{Top}}

\newcommand{\functor}[1]{\operatorname{\msans{#1}}}
\newcommand{\fixed}{\functor{Fixed}}

\newcommand{\clDec}{\mf{c}}

\newcommand{\deRham}{H_{\text{deRham}}}
\newcommand{\forms}[1]{\ms{F}^{#1}}
\newcommand{\clForms}[1]{\forms{#1}_\clDec}
\newcommand{\hol}{\mc{O}} 
\newcommand{\curDeg}[1]{\ms{C}^{#1}}
\newcommand{\curDim}[1]{\ms{C}_{#1}}

\newcommand{\shHone}[2]{\ensuremath{H^1(#1,\sheaf{#2})}}
\newcommand{\hbundles}[1]{H^1(#1,\pluriharmonic)} 
\newcommand{\shsub}[1]{_{\sheaf{#1}}} 
\newcommand{\shGamma}[1]{\Gamma(\sheaf{#1})} 
\newcommand{\mGamma}[1]{\Gamma {#1}}
\newcommand{\mshGamma}[2]{\mGamma{{#1}\shsub{#2} }}
\newcommand{\mshHone}[2]{\ensuremath{H^1{#1}\shsub{#2}}}

\newcommand{\onto}{\twoheadrightarrow}

\newcommand{\G}{\mc{G}}
\newcommand{\greens}{\mf{g}}

\newcommand{\lto}[1]{\overset{#1}{\to}} 
\newcommand{\lift}{\check}
\newcommand{\hollift}{\breve}

\newcommand{\positive}{\R_{>0}}
\newcommand{\positiveK}{\K_0}

\newcommand{\nempty}{{\not=\emptyset}}
\newcommand{\K}{\mathbb{K}}

\newcommand{\sca}{\mb{K}} 
\newcommand{\interior}{\operatorname{int}}
\newcommand{\diam}{\operatorname{diam}}
\newcommand{\divisor}{\msans}
\newcommand{\current}{\msans}

\newcommand{\maxMult}{\Upsilon}

\newcommand{\smear}{\mc{S}}
\newcommand{\dvect}[1]{\frac{\partial}{\partial #1}}

\newcommand{\nimble}{\ms{N}} 
\newcommand{\lenient}{\ms{L}} 
\newcommand{\comassNorm}[1]{\|{#1}\|_{\text{comass}}}
\newcommand{\supNorm}[1]{\|{#1}\|_{\text{sup}}}
\newcommand{\pder}[2]{\frac{\partial #1}{\partial #2}}
\newcommand{\cat}{\mc{S}}
\newcommand{\holder}{H}

\begin{document}

\title{Dynamical Objects for Cohomologically Expanding Maps.}
\author{John W. Robertson}
\maketitle

\bigskip
\noindent{John W. Robertson\footnote{Research partially supported by a Wichita State University ARCS Grant.}} \\
\noindent{Wichita State University} \\
\noindent{Wichita, Kansas 67260-0033} \\
\noindent{Phone: 316-978-3979} \\
\noindent{Fax: 316-978-3748} \\ 
\noindent{robertson@math.wichita.edu} \\

\begin{abstract}
The goal of this paper is to construct invariant dynamical
objects for a (not necessarily invertible) smooth self map
of a compact manifold. We prove a result 
that takes advantage of differences in rates of expansion in the terms
of a sheaf cohomological long exact sequence to create 
unique lifts of finite dimensional invariant subspaces of one term of the
sequence to invariant subspaces of the preceding term.
This allows us to take invariant cohomological classes
and under the right circumstances construct unique
currents of a given type, including unique measures of a given type, 
that represent those classes and are invariant
under pullback. A dynamically interesting self map may have a plethora
of invariant measures, so the uniquess of the constructed
currents is important. It means that if local growth is
not too big compared to the growth rate of the cohomological
class then the expanding cohomological class gives sufficient
``marching orders'' to the system to prohibit the formation
of any other such invariant current of the same type (say from some local
dynamical subsystem).
Because we use subsheaves of the sheaf of currents we 
give conditions under which a subsheaf will
have the same cohomology as the sheaf containing it. Using a smoothing
argument this 
allows us to show that the sheaf cohomology of the currents
under consideration can be canonically identified with the
deRham cohomology groups.
Our main theorem can be applied in both the smooth
and holomorphic setting.
\end{abstract}


\medskip



%
%

\section{Introduction}

Our purpose is to construct invariant dynamical objects
for a self map $f\colon X\to X$ of a compact topological space.
We make use of sheaf cohomology and differences in rates 
of expansion in different terms of
a long exact sequence to construct invariant sections of a sheaf.
We will show that there are invariant degree $1$ currents (or eigencurrents)
corresponding to each expanding eigenvector
of $H^1(X,\R)$. We also show that successive preimages
of sufficiently regular degree one currents converge to 
one of these eigencurrents.
We show that if most of the expansion $f\colon X\to X$ is ''along''
an invariant cohomological class $v\in H^k(X,\R)$ then
there is an invariant current $c$ in that cohomology class
and other sufficiently regular currents in the same class
converge to $c$ under successive pullback.

The group cohomology
of $\Z$ acting on a space of functions on $X$ via pullback
has been studied in the context of dynamical 
systems \cite{katok-combinatorial}. This work seems
related to ours, but to be pursued in an essentially
different direction. Our map $f$ is not assumed to be invertible, so
there is not necessarily a $\Z$ action, only an $\N$ action. 
Also, we use sheaves rather than functions and make
substantial use of cohomological tools. Most importantly,
we are particularly interested in the construction of invariant 
currents, especially when the current is some sense unique.

Our results are actually motivated by results in 
higher dimensional holomorphic dynamics showing the existence of 
a unique closed positive $(1,1)$ current
under a variety of circumstances (just about any
recent paper on higher dimensional holomorphic dynamics
either proves such results or makes essential use 
of such results, see e.g. 
\cite{forn:henon},
\cite{hubbard_oberste-vorth:henon1}, \cite{hubbard_oberste-vorth:henon2},
\cite{bedfordsmillie1}, \cite{bedfordsmillie2}, \cite{bedfordsmillie3}, 
\cite{bedfordlyubichsmillie4}, \cite{bedfordsmillie5}, \cite{bedfordsmillie6}, 
\cite{bedfordsmillie7}, 
\cite{cantat}, \cite{mcmullen:k3},
\cite{forn:cdhd}, \cite{forn:cdhd1},
\cite{forn:cdhd2}, \cite{forn:examples}, \cite{forn:classification},
\cite{jonsson:example}, \cite{jonsson-favre}, 
\cite{ueda:fatou}, \cite{ueda:normfam}, \cite{ueda:fatmaps},
and \cite{dinhsibony-greenCurrents}).

While invariant measures have been a focal point in dynamics, it seems
that invariant currents also have an imporant role to play. We will
show under mild conditions that if some degree one cohomological class 
of a smooth self map $f$ of a compact manifold 
is invariant and expanded there is necessarily a invariant
degree one current of a certain type representing that class. We obtain analogous
results for higher degree currents given bounds on the local growth rates
of $f$. The uniqueness of these classes is significant. It seems clear that one could 
modify a map locally near a fixed point to obtain other invariant currents of
the same type without affecting the topology. Thus our results also say that
any local modification that created an invariant current of the given type
{\it must violate the local growth conditions}. In other words, as long as
things do not grow too fast compared to the growth rate of the cohomology class,
the expansion of the cohomology class gives sufficient ``marching orders'' to
points that no other invariant cohomological class of the given type
can be created by purely local dynamical behavior. Our results give explicit
conditions under which uniqueness is guaranteed.
For degree one currents, no restriction on local growth rates is necessary
for our results. 

\section{Cohomomorphisms}

We will make use of sheaves in this paper. There are two standard
definitions of sheaves on a topological space $X$, 
one as a topological space (\cite{bredon},\cite{cas}),
and one as a functor on the category $\topcat_X$ satisfying various axioms 
(\cite{hartshorne},\cite{weibel}). 
Since we will often want to make use of a topology on sections of a sheaf
$\sheaf{A}$ that differs from the topology these inherit 
using the topological definition of a sheaf,
we will instead use the functor definition of a sheaf.

Our sheaves will always be sheaves of $\sca$ modules over some 
fixed field $\sca$. We will require that $\sca$ have an absolute
value for which $\sca$ is complete. 

Given a continuous map $f\colon X\to Y$ and sheaves $\sheaf{A}$ and
$\sheaf{B}$ on $X$ and $Y$ respectively, an $f$-cohomomorphism
is a generalized notion of a pullback from $\sheaf{B}$ to $\sheaf{A}$
through $f$. Different types of geometric objects pull back differently,
and this allows us to handle all cases at once.

We take the following facts from from \cite{bredon} page 14--15.

\begin{definition}
If $\sheaf{A}$ and $\sheaf{B}$ are sheaves on $X$ and $Y$
then an ``f-cohomomorphism'' $k\colon \sheaf{B}\to\sheaf{A}$
is a collection of homomorphisms $k_U\colon \sheaf{B}(U)\to \sheaf{A}(f^{-1}(U))$,
for $U$ open in $Y$, compatible with restrictions.
\end{definition}

Note that if $\sheaf{A}$ is a sheaf on $X$ and $f\colon X\to Y$
is continuous then there is a canonical cohomomorphism
$f_*\sheaf{A} \mto \sheaf{A}$ where $f_*\sheaf{A}$ is the direct
image of $\sheaf{A}$, i.e. given an open $U\subset Y$,
$f_*\sheaf{A}(U) = \sheaf{A}(f^{-1}(U))$.

\begin{remark}
Given a continuous 
map $f\colon X\to Y$ of topological spaces $X$ and $Y$ and
sheaves $\sheaf{A}$ and $\sheaf{B}$ on $X$ and $Y$ respectively,
all $f$-cohomomorphisms $f\colon \sheaf{B}\mto\sheaf{A}$ are
given by a composition of the form 
\[\sheaf{B}\overset{j}{\to}f_*\sheaf{A}\overset{f_*}{\to}\sheaf{A}\]
where $j\colon \sheaf{B}\to f_*\sheaf{A}$
is a sheaf homomorphism, and each such composition is seen to 
given an $f$-cohomomorphism.
\end{remark}

The usual notion of ``a morphism of sheaves on $X$'' is the same
as an $\id{X}$ cohomomorphism of sheaves on $X$.

\subsection{Cohomomorphisms and $\Gamma$.}

The functor $\Gamma$ returns the global sections of 
that sheaf. 
Given a morphism $\phi\colon\sheaf{A}\to\sheaf{A'}$ of sheaves on $X$,
$\mGamma{\phi}$ is just the homomorphism 
$\sheaf{A}(X)\to\sheaf{A'}(X)$. Given sheaves $\sheaf{A}$ and $\sheaf{B}$
on $X$ and $Y$ and given $f\colon X\to Y$ continuous then 
for a sheaf cohomomorphism $F\colon \sheaf{B}\to\sheaf{A}$ one
defines $\mGamma{F}$ to be the homomorphism $\sheaf{B}(Y)\to\sheaf{A}(X)$.
This extends $\Gamma$ to be a functor on the category of topological
spaces with an associated sheaf where morphisms are given by cohomomorphisms.

\section{Invariant Global Sections}

Fix a continuous self map $f\colon X\to X$ of 
a topological space $X$. We will be interested 
in $f$ self cohomomorphisms of sheaves $\sheaf{A}$ on $X$.
As we will typically have several sheaves of interest on $X$,
each with a corresponding $f$ self cohomomorphism,
we let $f\shsub{A}\colon \sheaf{A}\mto\sheaf{A}$
be the default notation for an $f$-cohomomorphism of $\sheaf{A}$.

Assume that $X$ is a manifold and that 
\[\sheaf{A}\overset{p}{\to}\sheaf{B}\overset{q}{\to}\sheaf{C}\]
is a short exact sequence of sheaves on $X$. 
Let $f\colon X\to X$ be a continuous self map of $X$ 
and assume further
that we are given $f$ self cohomomorphisms of each of these
sheaves and that
\begin{equation}
\label{shortexact}
\xymatrix{ {\sheaf{A}} \ar[r]_p \ar[d]^{f\shsub{A}} 
& {\sheaf{B}} \ar[r]_q \ar[d]^{f\shsub{B}} 
& {\sheaf{C}} \ar[d]^{f\shsub{C}} \\
{\sheaf{A}} \ar[r]_p 
& {\sheaf{B}} \ar[r]_q 
& {\sheaf{C}} \\
}
\end{equation}
commutes. We will say that a commutative diagram as in~\eqref{shortexact} is
an $f$ self-cohomomorphism of the sequence $\sheaf{A}\to\sheaf{B}\to\sheaf{C}$.

Applying the functor $\Gamma$ to this diagram, the rows can be extended in 
the usual long exact sequence. The resulting diagram
is commutative (\cite{bredon} page 62).

\begin{equation}
\label{longexact}
\xymatrix{ {0} \ar[r] 
& {\sheaf{A}(X)} \ar[r]_{\mGamma{p}} \ar[d]^{\mshGamma{f}{A}}
& {\sheaf{B}(X)} \ar[r]_{\mGamma{q}} \ar[d]^{\mshGamma{f}{B}}
& {\sheaf{C}(X)} \ar[r]_{\delta} \ar[d]^{\mshGamma{f}{C}} 
& {\shHone{X}{A}} \ar[r]_{H^1p} \ar[d]^{\mshHone{f}{A}} 
& {\dotsb}\\
{0} \ar[r] 
& {\sheaf{A}(X)} \ar[r]_{\mGamma{p}} 
& {\sheaf{B}(X)} \ar[r]_{\mGamma{q}} 
& {\sheaf{C}(X)} \ar[r]_{\delta} 
& {\shHone{X}{A}} \ar[r]_{H^1p} 
& {\dotsb }\\
}
\end{equation}

One can
think of $\sheaf{B}$ as providing local potentials for 
members of $\sheaf{C}$ and of $\sheaf{A}$ as being those
potentials which give rise to the zero member of $\sheaf{C}$.
It will be assumed that the reader is familiar with interpreting
$\shHone{X}{A}$ as classifying equivalence classes of 
bundles with transition functions in $\sheaf{A}$. 
We will frequently refer to members of $\shHone{X}{A}$
as bundles. Sections
of such bundles will be assumed to be given locally by local
sections of $\sheaf{B}$,
so that every member $c$ of $\shGamma{C}$ is given locally by
potentials in $\sheaf{B}$, and these potentials, taken
together, are a section of the 
corresponding bundle $\delta(c)\in \shHone{X}{A}$.

\begin{convention}
We will frequently refer to a member $v$ of $\shHone{X}{A}$
as a {\it bundle}, to a member $c\in\shGamma{C}$
as a {\it divisor} and if $\delta(c)=v$ we will
call $c$ a {\it divisor of the bundle $v$}. We think
this substantially adds to the readability of the paper.
\end{convention}

\begin{definition}
The {\it support} of a divisor $c\in\shGamma{C}$
is defined 
to be the complement of the union of all open sets $U$
such that $c\rest{U}=0$.
\end{definition}

\begin{lemma}
If an open set 
$U$ lies outside the support of some $c\in \shGamma{C}$
then $f^{-1}(U)$ lies outside the support of $f\shsub{C}(c)$
\end{lemma}
\begin{proof}
We note that by the definition of an $f$-cohomomorphism
$f\shsub{C}\colon \sheaf{C}\to\sheaf{C}$, 
since the cohomomorphism $f\shsub{C}$ on $\sheaf{C}(U)$ 
is a homomorphism from $\sheaf{C}(U)$ to 
$\sheaf{C}(f^{-1}(U))$ and the induced action
of $f\shsub{C}$ on $\shGamma{C}$ restricted to $U$
must agree with its action $\sheaf{C}(U)\to\sheaf{C}(f^{-1}(U))$, 
then if an open set $U$ is outside the support of $c$ then $f^{-1}(U)$
is outside the support of of $f\shsub{C}(c)$.
\end{proof}

The following conditions for a given 
$v\in \shHone{X}{A}$ will be of interest:

\begin{definition}
We will refer to a bundle $v\in \shHone{X}{A}$ for which $(H^1p)(v)=0$
as being {\it closed}. 
\end{definition}

Note that this notion depends upon the exact
sequence $\sheaf{A}\to\sheaf{B}\to\sheaf{C}$, and not just on $v$.
If $\sheaf{B}$ is $\gamma$ acyclic then every member of $\shHone{X}{A}$ is closed.

\begin{definition}
\label{noBasePoints}
We will call a bundle $v\in \shHone{X}{A}$ {\it base point free} 
if for every $x\in X$ there is some divisor $c\in \shGamma{C}$ associated
to $v$ whose support does not contain $x$.
\end{definition}

\begin{lemma}
\label{softImpliesNoBasePoints}
If $\sheaf{B}$ is soft, $X$ is a regular topological space, 
and $a\in\shHone{X}{A}$ is a closed bundle then $a$ is base point free.
\end{lemma}
\begin{proof}
From the long exact sequence
there is some $c'\in \shGamma{C}$ with $\delta(c')=a$
and given any point $x\in X$, from the fact that 
$\sheaf{B}\overset{q}{\onto}\sheaf{C}$ the germ $c'_x$ of $c'$ at $x$
is the image under $q_x$ of some germ $b''_x$ of $\shGamma{B}$ at $x$.
Choose an open neighborhood $U$ of $x$ on which there is 
some $b'\in \sheaf{B}(U)$ with $b'_x=b''_x$. The topological
assumption on $X$ implies that there is a neighborhood $V\Subset U$
of $x$. The fact that $\sheaf{B}$ is soft implies there is
some $b\in\shGamma{B}$ such that $b\rest{{\overline{V}}}=b'\rest{{\overline{V}}}$.
Then $c=c'-b\in\shGamma{C}$ has $\delta(c)=a$ and $x\not\in\supp(c)$.
\end{proof}

\begin{definition}
We will refer to a bundle $a\in \shHone{X}{A}$ such that $f\shsub{A}(a)=\lambda\cdot a$
for some $\lambda\in \C$ as a $\lambda$ {\it eigenbundle}. 
\end{definition}

We also find it useful to introduce a relevant notion of expansiveness
of a map $f\colon X\to X$
relative to a base point free closed eigenbundle $v\in\shHone{X}{A}$.

\begin{definition}
Given a base point free closed eigenbundle $v\in \shHone{X}{A}$
then we say that $f$ is cohomologically
expansive at $x$ for $v$ if for any open neighborhood $U$ of $x$
and any divisor $c\in\shGamma{C}$ of $v$, the set
$U$ intersects the support of $f\shsub{C}^k(c)$
for all sufficiently large $k$. 
\end{definition}

\begin{remark}
It is a corollary of the definition that the set of points at which
$f$ is cohomologically expansive for $v$ is closed and forward invariant.
If $\supp f\shsub{C}^k(c)=f^{-k}(\supp(c))$ for each $c\in \shGamma{C}$
then the set of cohomologically expansive points is totally invariant.
\end{remark}

The notion of being cohomologically expansive at $x$ for $v$ means 
roughly that under iteration by $f$
small neighborhoods $U$ of $x$ always grow to cover enough of $X$ 
that the pullback of the bundle $v$ to the set $f^k(U)$
is a nontrivial bundle on $f^k(U)$ whenever $k$ is large. 

We show that if $\sheaf{B}$ is soft and $X$ is a compact metric space
then some minimal expansion takes place at points
where $f$ is cohomologically expansive for a closed
eigenbundle $a\in \shHone{X}{A}$. 

We use $B_\epsilon(x)$ to denote
the ball of radius $\epsilon$ about $x$.

\begin{lemma}
\label{ceImpliesNotEquicontinuous}
Let $X$ be a compact metric space.
If $\sheaf{B}$ is soft and $v$ is a closed eigenbundle
then there exists $\delta > 0$ such that for every $\epsilon > 0$
there exists some $K > 0$ such that if $f$ is cohomologically
expanding at $x$ then for every $k > K$, $\diam f^k(B_\epsilon (x)) > \delta$.
\end{lemma}
\begin{proof}
The bundle $v$ is base point free by Lemma~\ref{softImpliesNoBasePoints}.
Using compactness we can conclude that there is a finite open
cover $U_1,\dotsc,U_\ell$ of $X$ such that for each $j$, $U_j$ is disjoint
from $\supp c_j$ for some $c_j\in\shGamma{C}$ with $\delta(c_j)=v$.
We will prove the lemma by contradiction. Let $\delta$
be the Lebesgue number of the cover $U_1,\dotsc,U_\ell$.
If the lemma is false there is some $\epsilon > 0$ and
some increasing sequence $k_n$ and points
$x_n$ at which $f$ is cohomologically expansive such that
$\diam f^{k_n}(B_\epsilon (x_n))\leq \delta$ for each $n$.
By going to a subsequence if necessary we can assume
$x_n$ converges to a point $x_\infty$. Letting $U=B_{\frac{1}{2}\epsilon} (x_\infty)$
we see that $U\subset B_\epsilon (x_n)$ for all large $n$
and thus there is some one $c_j$ of $c_1,\dotsc,c_\ell$ such
that $f^{k_n}(U)$ is disjoint from $\supp c_j$ for infinitely
many values of $n$. Consequently $U$ is disjoint from
$\supp f\shsub{C}^{k_n}(c_j)$ for infinitely many $n$,
contrary to $x_\infty$ being a point at which $f$ is 
cohomologically expansive for $v$.
\end{proof}

We included Lemma~\ref{ceImpliesNotEquicontinuous} to
show that our notion of cohomological expansion is
genuinely expansive. However, depending on the nature of
$\sheaf{A}$, being cohomologically expansive can imply that
neighborhoods grow a great deal under iteration indeed.  
In Lemma~\ref{ceExpands} we show that given any 
closed set $K$ such that the pullback of a fixed point free closed eigenbundle
$a\in\shHone{X}{A}$ to $K$ is a trivial bundle then any 
neighborhood $U$ of a point at which
$f$ is cohomologically expanding for $a$ is so expanded under
iteration that $f^k(U)\not\subset \interior K$ for all sufficiently large $k$.
The collection of such sets $K$ typically contains very large
sets about every point so no matter where $f^k(x)$ is the 
conclusion that $f^k(U)$ does not lie in any $\interior K$
implies some points of $f^k(U)$ must lie far away from $f^k(x)$.
The point is roughly that large iterates of any
neighborhood of $x$ can not be homotopically 
contracted to a point in $X$.

\begin{lemma}
\label{ceExpands}
If $\sheaf{B}$ is soft, then for any closed set $K\subset X$ such that the
image of $H^1(X,\sheaf{A})\to H^1(K,\sheaf{A}\rest{K})$ is zero, 
given any divisor $c\in \shGamma{C}$, there is another
divisor $c'\in\shGamma{C}$ associated to the same bundle
and $c'$ is supported outside the interior of $K$. Consequently,
if $f$ is cohomologically expansive
at $x\in X$ for some base point free closed eigenbundle 
$a\in \shHone{X}{A}$ then necessarily 
for any neighborhood $U$ of $x$, 
$f^k(U)\not\subset \interior K$
for all large $k$, where $\interior K$ is the interior of $K$.
\end{lemma}
\begin{proof}
We use the commutative diagram 
\[
\xymatrix{
 {H^0(X,\sheaf{B})} \ar[r]^{\mGamma{q}} \ar[d] & {H^0(X,\sheaf{C})} \ar[r]^\delta \ar[d] & {\shHone{X}{A}} \ar[d] \\
 {H^0(K,\sheaf{B}\rest{K})} \ar[r]^{\mGamma{q}} & {H^0(K,\sheaf{C}\rest{K})} \ar[r]^\delta & {0} 
}
\]
which we have written using $H^0$ instead
of $\Gamma$ so it is clear what the ambient space is in each case.
From exactness there exists some 
$\beta\in H^0(K,\sheaf{B}\rest{K})$ such that 
$\delta(\beta)=c\rest{K}$. Then since $\sheaf{B}$ is soft the map
$H^0(X,\sheaf{B})\to H^0(K,\sheaf{B}\rest{K})$ is surjective
so there is some $b\in\shGamma{B}=H^0(X,\sheaf{B})$ such that $b\rest{K}=\beta$.
Then $c'=c-(\Gamma q)(b)$ has $\delta(c')=\delta(c)$
and $c'\rest{K}=0$ so $\supp(c')$ is disjoint from the interior of $K$.

It is easy to see that if $f$ is cohomologically expansive
at $x\in X$ for some fixed point free closed eigenbundle 
$a\in \shHone{X}{A}$ then necessarily 
for any neighborhood $U$ of $x$, 
$f^k(U)\cap \supp c\nempty$ for all
large $k$ for any $c\in\shGamma{C}$ such that
$\delta(c)=a$. Hence $f^k(U)$ can not lie in the interior
of $K$ for any large $k$.
\end{proof}

\begin{convention}
We let $\K$ be 
either $\R$ or $\C$, although our central theorems only
require $\K$ to be a complete field with an absolute value. 
\end{convention}

The following Theorem takes advantage of the 
fact that
in an exact sequence the eigenvalues of members of nonadjacent
members of the sequence do not have to agree to give conditions under
which one can
uniquely ``lift'' fixed members of
one term of the exact sequence to a fixed member of the preceding term. 
Interpreted as a statement in the context of sheaf cohomology we
will be able to use this Theorem to make dynamical conclusions.

The theorem shows that each closed eigenbundle of the 
induced map $f\shsub{A}\colon\shHone{X}{A}\to\shHone{X}{A}$ 
with sufficiently large eigenvalue has
a unique associated invariant divisor $c\in\shGamma{C}$. 

\begin{definition}
Given any finite dimensional $\K$ vector space $V$
along with a linear map $g\colon V\to V$
and any positive real number $r$,
we let the $r$ chronically expanding subspace of $V$ be 
the span of the
subspaces associated\footnote{Meaning for each eigenvector
$\lambda$ we include not just the $\lambda$ eigenspace,
but also every $v\in V$ such that $(g-\lambda\cdot\id{V})^n(v)=0$
for some positive integer $n$.} 
to eigenvalues of absolute value greater than $r$. We refer to the $1$ chronically
expanding subspace simply as the chronically expanding subspace. 
\end{definition}

\begin{theorem}[Unique Invariant Subspace Theorem]
\label{main}
We will assume the following:
\begin{itemize}
\item $f\colon X\to X$ is a continuous self map of a topological space $X$.
\item We are given an $f$ self cohomomorphism of 
a short exact sequence of sheaves on $X$, 
\[\sheaf{A}\overset{p}{\to}\sheaf{B}\overset{q}{\to}\sheaf{C}\]
\item $\shGamma{B}$ is a Banach space over $\K$,  
and there exists some $\alpha,d\in \positive$ such that
$\|\mshGamma{f}{B}^k(\divisor{B})\|\leq d \cdot\alpha^k\| \divisor{B}\|$ 
for $k\in\N$, $\divisor{B}\in \shGamma{B}$,
\item $\shGamma{C}$ is a topological vector space over $\K$.
\item If a sequence $\divisor{C}_i\in \shGamma{C}$ of divisors 
converges to another divisor $\divisor{C}_\infty$ then the support of 
$\divisor{C}_\infty$ is contained in the closure of the 
union of the supports of $\divisor{C}_i$.
\item The maps $\mshGamma{f}{C}$ and $\mGamma{q}$ are continuous. 
\item We are given a finite dimensional $H^1(f\shsub{A})$ invariant subspace
$W$ of the $\alpha$ chronically expanding subspace of
$\shHone{X}{A}$. We also require $W$ to be comprised only of closed bundles.
\end{itemize}
Then given any $\K$ linear map $s\colon W\to\shGamma{C}$ such that $\delta s=\id{W}$
there is  a $\K$ linear map $\tau\colon W\to\shGamma{B}$ satisfying
\begin{equation}
\label{iterationFormula}
\kappa :=\lim_{k\to\infty} (\mshGamma{f}{C})^ksg^k=s+(\Gamma q)\tau
\end{equation}
where $g\colon W\to W$ is the inverse of $H^1f\shsub{A}\rest{W}$.
Under iterated pullback the rescaled pullbacks of any divisor
$\divisor{C}\in\shGamma{C}$ of a bundle $w\in W$ converge toward
the invariant plane of divisors $\kappa(W)\subset\shGamma{C}$.
The map $\kappa\colon W\to\shGamma{C}$ is the unique map making the
diagram
\[
\xymatrix@=12pt{
& & &   &   &  {W}  \ar '[d] [ddd] ^{g}   \ar [rd] ^{\iota}  \ar@{-->}[lld]_{\kappa} & \\
{\shGamma{B}} \ar[rrr]_{\mGamma{q}} \ar[ddd]_{\mshGamma{f}{B}} & & &  {\shGamma{C}}  \ar [ddd]_{\mshGamma{f}{C}}   \ar [rrr] _(.4){\delta}  &  &  &  {\shHone{X}{A}}  \ar [ddd] ^{H^1 f\shsub{A}} \\
\\
& & &  &   &   {W}   \ar [rd] ^{\iota} \ar@{-->}[lld]_{\kappa} \\
{\shGamma{B}} \ar[rrr]_{\mGamma{q}} & & &   {\shGamma{C}}   \ar [rrr] _(.4){\delta}  &  &  &  {\shHone{X}{A}} 
}
\]
commute. Finally, for any basepoint free eigenbundle $v\in W$ the support of
the corresponding invariant divisor $\kappa(v)\in\shGamma{C}$
is contained in the set of points on which $f$ is cohomologically
expansive for $v$.
\end{theorem}
\begin{proof}
We note that $\delta\bigl((\mshGamma{f}{C})sg-s\bigr)=0$ and so 
there is a map $\sigma\colon W\to \shGamma{B}$ such that
$(\Gamma q)\sigma=(\mshGamma{f}{C})sg-s$. 

Define 
$\Phi\colon \Hom(W,\shGamma{B})\to \Hom(W,\shGamma{B})$ by
$\Phi(\sigma)=(\mshGamma{f}{B})\sigma g^{-1}$. We will show that
the sequence of maps $\Phi^k$ is exponentially
contracting on $\Hom(W,\shGamma{B})$. Fix a norm $\|\cdot\|$ on $W$. 
The assumption that $W$ lies in the
$\alpha$ chronically expanding subspace of $\shHone{X}{A}$
implies that there exists some $\beta > \alpha$ and some $c > 0$ such that
$\|g^{-k}(w)\|\leq c\beta^{-k}\|w\|$ for $k\in \N$, $w\in W$.
This with the assumption on the rate
of expansion of $\mshGamma{f}{B}$ easily implies that
\[\|\Phi^k(\phi)(w)\|=\|(\mshGamma{f}{B})^k(\phi(g^{-k}(w)))\| \leq cd\Bigl(\frac{\alpha}{\beta}\Bigr)^k\|\phi\|\cdot \|w\|\]
Thus $\Phi^k$ is an operator of norm 
no more than $cd\Bigl(\frac{\alpha}{\beta}\Bigr)^k$, where $\alpha < \beta$.

Letting $\tau_k=\sigma+\Phi(\sigma)+\Phi^2(\sigma)+\dotsb+\Phi^k(\sigma)$
then $\lim_{k\to\infty} \tau_k$ converges to some map $\tau$.
It is easily confirmed that $(\Gamma q)\tau_k=(\mshGamma{f}{C})^ksg^{-k}-s$.
Equation~\eqref{iterationFormula} then follows by continuity of $\Gamma q$.

The conclusions about the map $\kappa$ are easy consequences of its definition.

For the final conclusion note that if we just let $W$ be the
span of $v$ then we have already shown that if $\divisor{C}$
is the unique invariant member of $\shGamma{C}$ associated
to $v$ then for any divisor $c'\in\shGamma{C}$
satisfying $\delta(c')=v$ letting
$\lambda$ be the eigenvalue of $v$ we can write
$c'=\kappa(v)+(\Gamma q)(b)$ and equation~\ref{iterationFormula}
becomes
$(\mshGamma{f}{C})^k{c'}/\lambda^k=\kappa(v)+(\Gamma q)(\mshGamma{f}{B})^k b\lambda^k$
where the final term goes to zero as $k\to\infty$ (by our
assumptions on growth rates of $g^{-1}$ and $\mshGamma{f}{B}$).
Hence $(\mshGamma{f}{C})^k(c')/\lambda^k$ converges to $c=\kappa(v)$. 
If $U$ is any open
subset of $X$ and if the support of $c'$ is disjoint from 
$f^n(U)$ for arbitrarily large values of $n$,
then the support of $(\mshGamma{f}{C})^n(c')$ must be
disjoint from $U$ for arbitrarily large values of $n$.
Since, rescaled, these converge to $c$
then $U$ must lie outside the support of $c$.
\end{proof}

\begin{remark}
While we have not formally required $X$ to be compact,
the requirement that $\shGamma{B}$
be a Banach space makes this the main case in which 
Theorem~\ref{main} is apt to have interesting applications.
\end{remark}

Theorem~\ref{main} shows that among all members of $\shGamma{C}$
representing a cohomology class in $W$ there is a unique invariant linear
subspace which can be identified with $W$ and all other such
members of $\shGamma{C}$
are contracted to this invariant copy of $W$ in $\shGamma{C}$ 
under (rescaled) pullback.

\begin{corollary}
Assume that the hypothesis of Theorem~\ref{main}
are satisfied, and that
$g\colon W\to W$ is dominated by a single simple real
eigenvalue $r > 0$ with eigenvector $v$. Let $\divisor{C}\equiv\kappa(v)$
be the unique invariant divisor of $v$.
Then given a divisor $\divisor{C}'\in \shGamma{C}$ of 
any $w\in W$ the successive rescaled pullbacks
$f\shsub{C}^k(\divisor{C}')/r^k$ converge to 
a multiple (possibly zero) of $\current{C}$.
\end{corollary}
\begin{proof}
This is a direct consequence of 
equation~\eqref{iterationFormula}.
\end{proof}

The assumption that $g\colon W\to W$ is dominated by a single
simple real eigenvalue is meant to handle the most typical situation,
and is not an essential restriction.

\begin{remark}
Given that for a fixed $f\colon X\to X$ the 
category of $\cat{C}$ sheaves $\sheaf{A}$ on $X$
endowed with an $f$ self cohomomorphism $F$ is an abelian category with enough 
injectives, then the functor $\fixed\Gamma$ which
gives the fixed global sections of $\sheaf{A}$ under $F$ will be 
left exact and its right derived functors should be of dynamical
interest. In the case where $\sheaf{A}$ is a sheaf of functions
and $f$ is invertible this is just group cohomology with
the group $\Z$ acting on $\shGamma{A}$ and has been an object
of study for some time (see, e.g. \cite{katok-combinatorial}). We anticipate
studying the case of more general sheaves $\sheaf{A}$ and
the right derived functors of the composition $\fixed\Gamma$
in a future paper, including the case of endomorphisms.
\end{remark}

\subsection{Regularity and Positivity}

Typically our regularity results for the members invariant plane $\kappa(W)$
will be most easily described in terms of $\sheaf{B}$ rather than $\sheaf{C}$.
We therefore make the following definition.

\begin{definition}
Given a subsheaf $\sheaf{B}'\subset\sheaf{B}$ 
we will say a divisor $\divisor{C}\in\shGamma{C}$ has local $\sheaf{B}'$ potentials
if $\divisor{C}\in\Gamma(q(\sheaf{B}'))$. 
This is equivalent to requiring that about each point $x\in X$ there is
an open neighborhood $U$ and some $\divisor{B}'\in\sheaf{B}'(U)$ such that
$q(\divisor{B}')=\divisor{C}\rest{U}$.
\end{definition}

The proof of Theorem~\ref{main}
implicitly provides a method to prove regularity results for
members of the invariant plane $\kappa(W)$. 
We make this explicit as a corollary (of the proof).

\begin{corollary}
\label{regularity}
Assume we are given $f\colon X\to X$ and a short exact sequence
of sheaves $\sheaf{A}\lto{p}\sheaf{B}\lto{q}\sheaf{C}$
satisfying the hypothesis of Theorem~\ref{main}. 
Assume that $\sheaf{B}'$ is a subsheaf of $\sheaf{B}$
and that $\mshGamma{f}{B}(\sheaf{B}')\subset \sheaf{B}'$.
Let $\sheaf{C}'$ be the image of $\sheaf{B}'$ under 
$q\colon\sheaf{B}\to\sheaf{C}$.
Let $\sheaf{A}'\subset \sheaf{A}$ be the kernel of 
$q\colon \sheaf{B}\to\sheaf{C}'$.
Assume that the canonical map $H^1(X,\sheaf{A}')\to H^1(X,\sheaf{A})$
is injective.
Assume that there are basis members $w_1,\dotsc,w_k$
of $W$ with divisors each of which has local potentials in $\sheaf{B}'$.
Let $r$ be the the inverse of the absolute value of the largest
eigenvalue of $g^{-1}$ (so for all $j\geq 0$, $g^{-j}$ is an operator of norm no more
than $cr^{-j}$ for some $c > 0$) 
Finally assume that for any sequence 
of numbers $a_j, j=0,1,2,\dotsc$ 
such that $\abs{a_j}$ is no more than a constant times $r^{-j}$ as $j\to \infty$
then for $\current{B}\in\shGamma{B'}$ the exponentially decaying sequence
\begin{equation}
\label{geometricSeries}
a_0\, \current{B} + a_1 \,(\mshGamma{f}{B})(\current{B})+ a_2\,(\mshGamma{f}{B})^2(\current{B})+\dotsb
\end{equation}
converges in the Banach space structure on $\shGamma{B}$ to a member
of $\shGamma{B'}$. Then the map $\kappa\colon W\to\shGamma{C}$ 
lands in $\shGamma{C'}$.
\end{corollary}
\begin{proof}
Since $W$ lies in the $\alpha$ chronically expanding subspace of $W$
then necessarily $\alpha / r < 1$. Thus the terms of equation~\eqref{geometricSeries}
have exponentially decreasing norms and the series is exponentially decaying.

By the assumption of a divisor in $\shGamma{C'}$
for each member $w_j$ of a basis then the map $s\colon W\to \shGamma{C}$
in Theorem~\ref{main} can be assumed to land in $\shGamma{C'}$.
Then $(\mshGamma{f}{C})sg^{-1}-s$ lands in $\shGamma{C'}$ 
and satisfies $\delta((\mshGamma{f}{C})sg^{-1}-s)=0$.
Since $H^1(X,\sheaf{A'})\to H^1(X,\sheaf{A})$ injects it 
easily follows that  
for each $w_j$ one can choose $\sigma(w_j)$ to be a
member $\current{B}_j$ of $\shGamma{B}'$.
Using the basis $w_1,\dotsc,w_k$ to write $g^{-1}$ as a matrix
$A$, and letting $a_{ij,\ell}$ be the $ij$ entry of $A^\ell$
(so for each $ij$, $a_{ij,\ell}$ is bounded by a constant times $r^{-\ell}$)
we see that 
$\tau_\ell(w_j)=\current{B}_j 
+ (\mshGamma{f}{B})(a_{1j,1}\current{B}_1 + \dotsb + a_{kj,1}\current{B}_k) 
+ (\mshGamma{f}{B})^2(a_{1j,2}\current{B}_2+\dotsb + a_{kj,2}\current{B}_k) 
+ \dotsb 
+ (\mshGamma{f}{B})^\ell(a_{1j,\ell}\current{B}_1+\dotsb+a_{kj,\ell}\current{B}_k)$.
Gathering all the $\current{B}_1$ terms, $\current{B}_2$ terms, etc... from the right hand side
we see that $\tau=\lim_{k\to\infty} \tau_k$ is a member of $\shGamma{B'}$
and thus that $\kappa$ lands in $\shGamma{C'}$ by equation~\eqref{iterationFormula}.
\end{proof}

The following trivial observation will suffice for our needed positivity conclusions.

\begin{observation}
\label{positivity}
Assume we have an $f$ self cohomomorphism of
a short exact sequence of sheaves 
$\sheaf{A}\overset{p}{\to}\sheaf{B}\overset{q}{\to}\sheaf{C}$
satisfying the hypothesis of Theorem~\ref{main}, 
and also 
a subsheaf $\sheaf{C}'\subset\sheaf{C}$ such that
\begin{enumerate}
\item $\sheaf{C}'$ is closed under multiplication by $\positive$. 
Note that $\sheaf{C}'$ is not necessarily a sheaf of $\K$ modules, or even of groups.
\item $f\shsub{C}(\sheaf{C}')\subset \sheaf{C}'$
\item $\Gamma(\sheaf{C}')$ is closed in $\shGamma{C}$.
\end{enumerate}
Then for any closed eigenbundle
$v\in \shHone{X}{A}$ with eigenvalue in $\positiveK$ and at
least one divisor $\divisor{C}'\in\shGamma{C'}$
the unique invariant divisor $\divisor{C}\in\shGamma{C}$ of $v$
also lies in $\shGamma{C'}$.
\end{observation}
\begin{proof}
The proof is trivial since
$\divisor{C}=\lim_{k\to\infty} (\mshGamma{f}{C})^k(\divisor{C}')/\lambda^k$ 
where $\lambda\in\positive$ is the
eigenvalue of $v$.
\end{proof}

\section{Subsheaf Cohomology}

In applications of Theorem~\ref{main} it is common that there is
a well understood exact sequence of sheaves
\begin{equation}
\label{niceSheafSequence}
\sheaf{S}_0\lto{d_0}\sheaf{S}_1\lto{d_1}\sheaf{S}_2\lto{d_2}\dotsb
\end{equation}
and that $\sheaf{B}$ is a subsheaf of $\sheaf{S}_k$ for some $k$,
$\sheaf{A}$ is the kernel of ${d_k}\rest{\sheaf{B}}\colon\sheaf{B}\to\sheaf{S}_{k+1}$
and $\sheaf{C}$ is the image of $\sheaf{B}$ in $\sheaf{S}_{k+1}$.
Moreover, in these cases the self cohomomorphism $f$ on $\sheaf{A}\to\sheaf{B}\to\sheaf{C}$
is induced by an $f$ self cohomomorphism of the sequence~\eqref{niceSheafSequence}.
In order to apply Theorem~\ref{main} to these cases we need to
understand the $R$ module $H^1(X,\sheaf{A})$ and its induced self map.

There does not seem to be a computationally useful way to extract
an injective resolution of $\sheaf{A}$ using subsheaves of
$\sheaf{S}_0\lto{d_0}\sheaf{S}_1\lto{d_1}\dotsb$
even if this last sequence is acyclic. Consider for example the case
where for each $n$, $\sheaf{S}_n$ is the sheaf of currents of degree $n$
and $\sheaf{B}\subset\sheaf{S}_k$ is a subsheaf of mildly regular currents.
It is not clear one could make the regularization method of \cite{deRham} work to compare 
$H^1(X,\sheaf{A})$ to deRham cohomology groups because his chain homotopy operator
$A$ does not restrict well to $\sheaf{B}$ since $\ed A$ does
not preserve regularity. We use a standard sheaf cohomological
trick, which we include here as a proposition which we will need and 
which we expect to be commonly used in conjuction with Theorem~\ref{main}
because of the requirement that $\shGamma{B}$ be a Banach space.

\begin{theorem}[Subsheaf Cohomology]
\label{chainExchange}
Assume we are given an exact sequence of sheaves 
$\sheaf{S}_0\lto{d_0}\sheaf{S}_1\lto{d_1}\sheaf{S}_2\lto{d_2}\dotsb$
and that $\sheaf{B}$ is a subsheaf of $\sheaf{S}_k$ for some $k\geq 1$.
Let $\sheaf{A}=\ker{d_k}\rest{\sheaf{B}}$, and $\sheaf{B'}$ be
the preimage of $\sheaf{B}$ under $d_{k-1}$.
Further assume that 
for each $j \geq 1$ we have
$H^{j}(X,\sheaf{B'})=0$, $H^j(X,\sheaf{B})=0$ and for any $m$ satisfying $0\leq m\leq k-1$ we have $H^j(X,\sheaf{S}_m)=0$ for $j\geq 1$.
Then for each $n\geq 1$ there is a canonical isomorphism
\[H^n(X,\sheaf{A})\cong H^{n+k}(X,\ker d_0).\] 
\end{theorem}
\begin{proof}
While this result is essential for us, its proof is a standard
cohomological trick. First one notes that $\ker d_{k-1}\rest{\sheaf{B'}}=\ker d_{k-1}$
by the definition of $\sheaf{B'}$. One has the short exact sequences
of sheaves:
\[\ker d_{k-1}\to \sheaf{B'}\to (d_k(\sheaf{B'})=\sheaf{A})\]
and
\[\ker d_j\to \sheaf{S}_j\to \ker d_{j+1}, \quad j=0,\dotsc,k-2.\]
Considering the long exact sequences for these shows that the
induced maps $H^n(X,\sheaf{A})\to H^{n+1}(X,\ker d_{k-1})$
and $H^{n+j}(X,\ker d_{k-j})\to H^{n+j-1}(X,\ker d_{k-j-1})$
are isomorphisms for $j=1,\dotsc,k-1$.
Composing each of these canonical isomorphisms
gives a canonical isomorphism from $H^n(X,\sheaf{A})\to H^{n+k}(X,\ker d_0)$.
\end{proof}

\begin{remark}
We take it as clear from the functorality of the
$\delta$ map in the long exact sequence
that given an $f$-self cohomomorphism of 
$\sheaf{S}_0\lto{d_0}\sheaf{S}_1\lto{d_1}\sheaf{S}_2\lto{d_2}\dotsb$
which maps $\sheaf{B}$ to itself that
the induced map of $H^1(X,\sheaf{A})$ is identified
with the induced map of $H^{k+1}(X,\ker d_0)$
via the above isomorphism. 
\end{remark}

We will need one more tool be able to make effective
use of Theorem~\ref{chainExchange} for calculating
sheaf cohomology of subsheaves of sheaves of currents.

\begin{definition}
By an {\it interval flow} $h$ on a bounded open interval
$I\subset \R$ we will mean the flow
obtained by integrating a vector field
of the form $\sigma(t)\dvect{t}$ where $\sigma$ is
positive exactly on $I$ and zero elsewhere. We use
$h(x,t)$ to denote the location of $x\in R$ after 
following the flow for time $t$. 
\end{definition}

\begin{definition}
By an $n$-box in $\R^n$ we will mean an open subset
which is a product of $n$ bounded open intervals
$I_1,\dotsc,I_n$.
By an $n$-box in an $n$ dimensional manifold we will
mean an $n$-box which is compactly supported in some coordinate patch.
By an $n$-subbox of an $n$ box $U=I_1\times \dotsb\times I_n$
we will mean an $n$ box of the form $I'_1\times \dotsb\times I'_n$
where $I'_k$ is a subinterval of $I_k$ for each $k\in 1,\dotsc,n$.
\end{definition}

\begin{definition}
By an $n$-box flow we will mean the $\R^n$ action $h$ on 
$\R^n$ which is the product of $n$ interval flows
$h_1(t_1),\dotsc,h_n(t_n)$ on $\R^n$.
That is $h(x,t)=(h_1(x_1,t_1),\dotsc,h_n(x_n,t_n))$
where $x=(x_1,\dotsc,x_n)$, $t=(t_1,\dotsc,t_n)$
and $h_1,\dotsc,h_n$ are interval flows
on $I_1,\dotsc,I_n$ respectively. We refer
to the $n$-box $I_1\times\dotsb\times I_n$
as the {open support} of the $n$-box flow.
We will often $h_t$ to denote the 
diffeomorphism $h(\cdot,t)\colon\R^n\to\R^n$.
\end{definition}

\begin{definition}
Let $h$ be an $n$-box flow on an $n$-box $B$.
Let $\rho$ be a compactly
supported smooth volume form on $\R^n$. With this data
we define an operator $\smear_{h,\rho}$ on 
smooth $k$ forms on any $n$ box $U$ containing $B$ by
\begin{equation}
\label{smearForm}
\smear_{h,\rho}(\phi)=\int_{R^n} h_t^*(\phi) \rho(t)
\end{equation}
We say $\smear_{h,\rho}$ defines a {\it box smear} on $U$, or {\it smears} $U$.
We will omit the subscript
from $\smear_{h,\rho}$ when the meaning is clear from context.
It is clear $\smear(\phi)$ is compactly supported
in $U$ if $\phi$ is.
\end{definition}

It is clear from the definition of $\smear$ that if
$\psi$ is an $n-k$ form on $U$ then
\[\int_U \smear_{H,\rho}(\phi)\wedge \psi =\int_U \phi\wedge \smear_{-H,\rho}(\psi)\]
where $-H$
is the family $H_t$ with the parameter negated.
From this motivation we define a smear of a current.

\begin{definition}
Given $h,\rho$ defining a smear on an $n$ box $U$ we define the smear
$\smear_{h,\rho}$ on currents on $U$ via
\[<\smear_{h,\rho}(\current{C}),\phi>\equiv<\current{C},\smear_{-h,\rho}(\phi)>.\]
\end{definition}

\begin{lemma}
\label{smearFacts}
Given $h$, $\rho$ defining a smear $\smear$ on an $n$ box
$U$ then 
$\ed \bigl(\smear(\current{C})\bigr)=\smear(\ed \current{C})$
for currents $\current{C}$ on any open subset of $U$ containing the 
open support of the smear.
Also, restricted to the open support of the smear, $\smear(\current{C})$ 
is a smooth form on $V$.
\end{lemma}
\begin{proof}
We remark that it is clear that 
$\ed \bigl(\smear(\phi)\bigr)=\smear(\ed \phi)$
for forms $\phi$, and consequently for currents $\phi$ via the definition.

Because on the open support of the smear, a smear is just convolution with a smooth
function, then we see that if $V$ is an open subset of the open support
of smear $\smear$ on $U$ then for any current $\current{C}$ on $U$,
$\smear(\current{C})\rest{V}$ is a smooth form on $V$.
\end{proof}

\begin{proposition}
\label{sheafIsSoft}
Let $\sheaf{B}$ be a sheaf of degree $k$ currents.
Assume that $\sheaf{B}$ contains the sheaf of smooth
$k$ forms on $X$, 
and that
$\sheaf{B}(U)$ is closed under smears on any $n$-box $U\subset X$.
Let $\sheaf{B'}$ be the preimage under $\ed$ of 
$\sheaf{B}$ in the sheaf of degree $k-1$ currents.
Then $\sheaf{B'}$ is soft, and therefore, $\Gamma$-acyclic.
\end{proposition}
\begin{proof}
To show that $\sheaf{B'}$ is soft it is sufficient to show
that $\sheaf{B'}$ is locally soft (\cite{bredon} page 69).
Given an $n$-box $U$ in $X$ we therefore only need to show
that if $K$ is a closed subset of $X$ in $U$
and if $W$ is an open neighborhood of $K$
then given any member $\current{B}_0'$ of $\sheaf{B'}(W)$
there is an open neighborhood $W_0\subset W$ of $K$
and a member $\current{B}'\in \sheaf{B'}(U)$ such that
$\current{B}'\rest{W_0}=\current{B}_0'\rest{W_0}$.

\begin{figure}
\label{postproductionCurrent.jpg}
\centerline{\includegraphics[width=3in,height=3in]{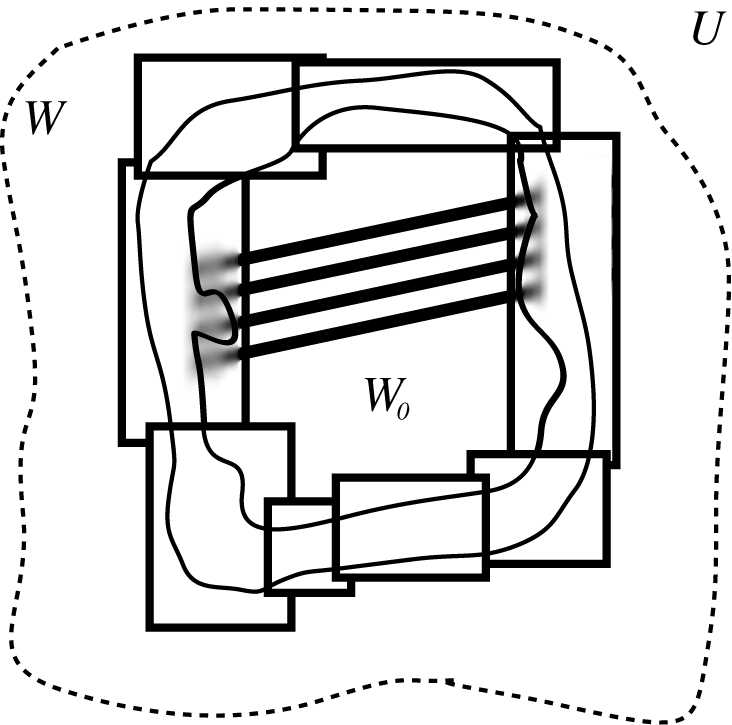}} 
\caption{A current comprised of parallel submanifolds smeared and cropped.}
\end{figure}

Choose any pair of open sets $V_1,V_2$
such that $K\Subset V_1\Subset V_2\Subset W$.
Then $\overline{V_2}\setminus V_1$ is compact
and can therefore be covered by finitely many (open)
$n$-subboxes $Y_1\dotsc,Y_N$ of $U$. Moreover
these subboxes can all be chosen to be disjoint from
$K$ and to lie inside $W$. 
Letting $\smear_1,\dotsc,\smear_n$
be smears on $U$ with open support $Y_1,\dotsc,Y_N$ respectively then 
let $\current{B}=\smear_1(\smear_2(\dotsb(\smear_N(\current{\current{B}_0'}))\dotsb))$.
Then on each $Y_j$, $\current{B}$ is given by a smooth $k$ form.
Also, $\current{B}\rest{W}=\current{B}_0'\rest{W}$.
Finally, we choose a smooth function $\psi\colon U\to [0,1]$
which is one on a neighborhood of $\overline{V_1}$ and 
zero on a neighborhood of $U\setminus V_2$. Then
the current $\current{B}'\equiv\psi \current{B}$ 
extends (by zero) to a current on all of $U$. 
Then for each $Y_j$, $\current{B'}\rest{Y_j}$ is a smooth function
times a smooth form. Thus $\ed(\current{B'}\rest{Y_j})$ is a smooth
form and lies in $\sheaf{B}(Y_j)$.
The boxes $Y_j$ cover $\overline{V_2}\setminus V_1$.
Outside $V_2$, $\current{B'}$ is identically zero.
We know that $\ed\current{B}\in\sheaf{B}(W)$
by Lemma~\ref{smearFacts}. We also know that 
$\psi\equiv 1$ on an open neighborhood $W_1$ of $\overline{V_1}$.
Thus $\ed(\current{B'}\rest{W_1})=\ed(\current{B}\rest{W_1})\in\sheaf{B}(W_1)$.
We thus conclude that $\current{B'}\in\sheaf{B'}(U)$ since its restriction
to each $Y_j$, to $W_1$ and to $U\setminus\overline{V_2}$ is
a section of $\sheaf{B'}$.
Letting $W_0=V_2\setminus \overline{(Y_1\cup Y_2\cup\dotsb\cup Y_N)}$
then $W_0$ is an open neighborhood of $K$, then $W_0\subset W_1$ so
$\current{B'}\rest{W_0}=\current{B}\rest{W_0}=\current{B}_0'$
since $W_0$ is disjoint from the open support of each of the smears
$\smear_1,\dotsc,\smear_N$. This completes the proof that $\sheaf{B'}$ is 
soft.
\end{proof}

The following gives a broad generalization of the equalivalence of
the cohomology of currents with the deRham cohomology groups.
To the author's knowledge, this result is new.

\begin{corollary}
\label{subsheafCohomology}
Let $\sheaf{B}$ be a sheaf of degree $k$ currents.
Assume that $\sheaf{B}$ contains the sheaf of smooth
$k$ forms on $X$, 
and that
$\sheaf{B}(U)$ is closed under smears on any $n$-box $U\subset X$.
Letting $\sheaf{A}$ be the subsheaf of $\ed$ closed members
of $\sheaf{B}$, then 
\[H^m(X,\sheaf{A})=H^{m+k}(X,\K),\]
where $\K$ is $\R$ or $\C$ depending on whether
or not we allow complex valued currents and forms.
\end{corollary}
\begin{proof}
This is an immediate consequence of Proposition~\ref{sheafIsSoft}
and Theorem~\ref{chainExchange}.
\end{proof}

\section{Invariant Currents}

\begin{notation}
If $\sheaf{G}$ is some sheaf of functions on a smooth orientable
manifold $X$ we will use $\forms{k}(\sheaf{G})$ to denote
the sheaf of $k$ forms on $X$ with coefficients in $\sheaf{G}$.
We will let $\clForms{k}(\sheaf{G})$ be the subsheaf of closed (in the
sense of currents) members of $\forms{k}(\sheaf{G})$.
\end{notation}

It will be convenient to use either degree or dimension of a current
depending on the context (just as dimension and codimension
are useful for discussing manifolds), so we will not stick
to just one of these terms.
We will let $\curDeg{k}$ denote the sheaf of degree $k$ currents with
the index written above as is typical for cohomology since $\ed$ 
increases the degree. We will similarly write $\curDim{k}$
for the sheaf of dimension $k$ currents with the index written below
since $\ed$ decreases dimension as is common for homology.
We use the following convention to realize a form $\alpha$ as a current
so that if $\alpha$ is $C^1$ then $\ed \alpha$ is the same whether computed
as a current or a form.

\begin{definition}
Given an $k$ form $\alpha$ with $L^1$ coefficients 
on an $n$ manifold $X$ we realize $\alpha$ as a degree $k$ current via
\[\beta\mapsto (-1)^{\binom{k+1}{2}}\int_X \alpha\wedge \beta\]
\end{definition}

\begin{definition}
Given a  (possibly complex)
nonzero deRham cohomology class $c\in \deRham^k(X)$ with
$f^*(c)=\alpha \cdot c$ for some scalar $\alpha \in\C$
we will refer to a current $\current{C}$ in the same cohomology
class as $\alpha$ as an eigencurrent for $f$ if
$f^*(\current{C})=\alpha \current{C}$.
\end{definition}

Currents naturally pushforward, rather than pullback. Because
we are considering maps which are not necessarily invertible
we need to address how this pullback is performed.
If $f$ has critical points it is impossible
to define a continuous pullback operation $f^*$
on all currents in a way that agrees with expected cases. For instance, consider $f(x)=x^2$
and let $\current{C}_a$ be the dimension one current on $\R$
with $\current{C}_a(h(x)\ed x)=h(a)$, i.e. $\current{C}_a$ is a unit mass vector.
Then the pullback
$f^*(\current{C}_a)$ should be the sum of weighted unit masses at the two preimages of this vector
(just like the pullback of a point mass is a sum of point masses each weighted by multiplicity),
that is, 
$f^*(\current{C}_a)=\frac{1}{2\sqrt{a}}\bigl(\current{C}_{\sqrt{a}}-\current{C}_{-\sqrt{a}}\bigr)$. 
However, these pullbacks do not converge to a current as $a\to 0$ 
so $f^*(\current{C}_0)$ is not defined.
Since we want $f^*$ to be continuous, we are forced to work with currents
that have some extremely mild regularity. We address this in the next section.

\subsection{Nimble Forms and Lenient Currents}

Finding a good set of currents to use to study
smooth finite self maps (not necessarily invertible) of compact manifolds turns out to
be rather delicate. 
Our solution is to first expand our class of forms to
include pushforwards (in the sense of currents) of forms through 
an appropriate class of smooth maps. Then we restrict
our attention to currents which act
on this extended class of forms. 

This solution has the very nice property that it can potentially be
adapted directly to study the dynamics
of other various other categories of smooth maps (by simply changing
which forms are considered nimble, according to the class of maps used).
It will convenient to first define the natural
pushforward operator on forms:

\begin{definition}
Given a compact orientable manifold $X$
we let $\cat_X$ be the category of smooth  
maps $f\colon X\to X$ of nonzero
degree and having the property that the critical set has measure zero.
We use critical set here to mean the points at which $Df$ is not
invertible.
\end{definition}

It follows from our definition that the image of any set of positive
measure under some $f\in \cat_X$ has positive measure. 

\begin{definition}
\label{nimbleForms}
Given a compact orientable manifold $X$ we define $\nimble^k$ to be those currents $\varphi$ 
which are a finite sum of currents of the form $p_*(\sigma)$ where
$p\colon X\to X$ is a map in $\cat_X$ and $\sigma$ is a form of degree
$k$. The pushforward $p_*(\sigma)$ is computed in the sense of currents.
\end{definition}

We will later show that nimble forms are also, in fact, bona fide forms.

\begin{definition}
\label{nimbleTopology}
We topologize $\nimble^k$ by saying $\varphi_j\to \varphi$ in $\nimble^k$
if for sufficiently large $j$ there are maps $f_1,\dotsc,f_k$
and $k$ forms $\sigma_{1j},\dotsc,\sigma_{kj}$ as well as forms
$\sigma_1,\dotsc,\sigma_k$ such that 
$\sum_i f_{i*}(\sigma_{ij})=\varphi_j$ and 
$\sum_i f_{i*}(\sigma_i)=\varphi$
(where pushforwards are taken in the sense of currents) 
and for each $i\in 1,\dotsc,k$, the forms $\sigma_{ij}$
converge to $\sigma_i$ in the strong sense (i.e. 
all derivatives converge uniformly).
\end{definition}

\begin{lemma}
Given a compact orientable manifold $Y$, $\nimble^k(Y)$ is a topological
vector space.
\end{lemma}
\begin{proof}
This follows easily from our definition of the topology.
\end{proof}

We now define the corresponding space of currents.

\begin{definition}
We define the dimension $k$ {\it lenient currents} $\lenient_k(Y)$ 
to be the topological dual of $\nimble^k(Y)$. 
Every member of $\lenient_k(Y)$ is a dimension $k$ current, but
with the added structure of its action on all nimble $k$ forms.
We give $\lenient_k$ the weak topology, i.e. $\current{C}_i\to\current{C}$
in $\lenient_k$ iff $<\current{C}_i,\varphi>\to<\current{C},\varphi>$
for every $\varphi\in\nimble^k$.
We write $\lenient^k$ for the lenient currents of degree $k$.
\end{definition}

We define operations of wedge products with smooth forms as is usual
for currents.
It is clear that the lenient dimension $k$ currents give a sheaf
on $X$.

The following properties of nimble forms are also immediately clear.

\begin{lemma}
Let $f\colon X\to X$ be a member of $\cat_X$.
The pushforward (as a current) of a nimble $k$ form by $f$
is again a nimble form. Moreover 
$f_*\colon \nimble^k(X)\to\nimble^k(X)$ is continuous
(in the topology of nimble forms).
Also the exterior derivative of a nimble form (as a current)
is a nimble form and 
$\ed \colon \nimble^k(X)\to \nimble^{k+1}(X)$ is 
continuous. 
\end{lemma}

The basic necessary facts about pulling back lenient currents are
then immediate. We state them here:

\begin{lemma}
\label{pullbackBasics}
Given $f\colon X\to X$ a member of $\cat_X$ 
the induced map $f^*$ on the sheaf of lenient 
degree $k$ currents is an $f$ cohomomorphism of sheaves.
Both $f^*\colon \lenient^k(X)\to\lenient^k(X)$
and $\ed\colon \lenient^k(X)\to\lenient^{k+1}(X)$ are continuous.
Lastly, $f^*\ed =\ed f^*\colon\lenient^k(Y)\to\lenient^{k+1}(X)$.
\end{lemma}

\begin{proposition}
\label{pushforwardResults}
Assume that $f\colon X\to X$ is a member of $\cat_X$. Let $R$
be the regular set of $f$. By Sard's theorem $R$ has full measure.
Since the critical set is compact then $R$ is an open subset of $X$. 
Since the preimage of a measure zero set has measure zero for $\cat_X$ maps
then $f^{-1}(R)$ is also a full measure open set in $X$.
There is a well defined operation $f_\star$ which maps
$k$ forms on $f^{-1}(R)$ to $k$ forms on $R$. Given a $k$ form $\beta$ on $X$,
$f_\star(\beta)$ is defined on any open subset $V\subset R$ such that
each component $U_1,\dotsc,U_m$ of $f^{-1}(V)$ maps diffeomorphically onto $V$
by the formula
\begin{equation}
\label{pushforward}
f_\star(\beta)\rest{V}\equiv \frac{1}{\deg f}\sum_i\Bigl((f\rest{U_i})^{-1}\Bigr)^\star(\beta)\cdot \sigma_i
\end{equation}
where $\sigma_i\in\{\pm 1\}$ is the oriented degree of $f\rest{U_i}\colon U_i\to V$.
The pushforward $f_\star$ satisfies:
\begin{itemize}
\item $f_\star \ed=\ed f_\star$ (keeping in mind that $f_\star$ returns a current on $R$)
\item $f_\star(1)=1$
\item $f_\star(f^*(\beta)\wedge \alpha)=\beta\wedge f_\star(\alpha)$
\item $(f_\star)^n=(f^n)_\star$
\item The formula 
\begin{equation}
\label{fundamental}
\ds{\int_X f^*(\beta)\wedge\alpha = \int_X \beta\wedge f_\star(\alpha)}
\end{equation}
holds for any
$k$ form $\beta$ with $\bounded$ coefficients on $Y$ and any smooth $n-k$ form $\alpha$ on $X$.
This justifies using $f_\star$ to pullback currents. (Part of the conclusion is that both
sides are integrable.) 
\end{itemize}
\end{proposition}
\begin{proof}
Each statement is a consequence of formula~\eqref{pushforward} except the integrability
conclusion for equation~\eqref{fundamental}.
Local charts  can be given which are bounded subsets of $\R^n$
and for which $Df$ remains uniformly bounded (over each of the 
charts) and thus $f^*(\beta)$ will be a form with $\bounded$ coefficients
in these charts. Thus the left hand side of \eqref{fundamental} is 
the integral of a bounded function over a finite union of bounded charts
and is therefore absolutely integrable.
Since $f_\star(f^*(\beta)\wedge \alpha)=\beta\wedge f_\star(\alpha)$ 
it is sufficient to show that if $\gamma$ is an 
$n$ form with $\bounded$ coefficients then 
\begin{equation}
\label{whatToProve}
\int_{f^{-1}(R)} \gamma=\int_R f_\star(\gamma).
\end{equation}
Typicaly $f_\star(\gamma)$ is unbounded so we need to show
that the right hand side of \eqref{whatToProve} is integrable. About any point 
$x\in R$ we can find an open $V$ such that each of the preimages $U_1,\dotsc,U_k$
of $V$ is mapped diffeomorphically onto $V$. Since $X$ is orientable and $n$ dimensional there
is a well defined notion of the absolute value of an $n$ form.
Then
\[\int_V \abs{f_\star(\gamma)}\leq 
\frac{1}{\deg f}\sum_i \int_{V} \Big\arrowvert\Bigl((f\rest{U_i})^{-1}\Bigr)^\star(\gamma)\Big\arrowvert
=\sum_i \int_{U_i} \abs{\gamma}=\int_{f^{-1}(V)} \abs{\gamma}.\]
Now $R$ is covered by countably many such sets $V$ and listing them as
$V_0,V_1,V_2,\dotsc,$ we can let
$V'_0=V_0, V'_1=V_1\setminus V_0, V'_2=V_2\setminus (V_0\cup V_1),\dotsc$.
Then $R$ is the union of the countable collection of disjoint measurable
sets $V'_j$ and 
\[\int_R \abs{f_\star(\gamma)} = \sum_j \int_{V_j} \abs{f_\star(\gamma)}\leq 
\sum_j \int_{f^{-1}(V_j)} \abs{\gamma}=\int_{f^{-1}(R)} \abs{\gamma}.\]
Since $\int_{f^{-1}(R)} \abs{\gamma}$ is finite then 
$f_\star(\gamma)$ is an $L^1$ form. Using precisely the same
argument but with the absolute values removed and the inequalities
replaced with equalities then shows
$\int_R f_\star(\gamma)=\int_{f^{-1}(R)} \gamma$.
\end{proof}

Since $R$ and $f^{-1}(R)$ are open and full measure then
$f_\star$ is an operator which takes in forms on $X$ and returns forms 
defined almost everywhere on $X$.

We now show that nimble forms are bona fide forms. 

\begin{lemma}
\label{nimbleAreForms}
If $g\colon X\to X$ is a map in $\cat_X$ 
and $\sigma$ is a smooth $k$ form on $X$
then the current $g_*(\sigma)$
is the current of integration against the form $g_\star(\sigma)$.
\end{lemma}
\begin{proof}
If $\varphi$ is a smooth $n-k$ form then by definition
$<g_*(\sigma),\varphi>=<\sigma,g^*(\varphi)>=
(-1)^{\binom{k+1}{2}}\int_X \sigma\wedge g^*(\varphi)=
(-1)^{\binom{k+1}{2}}\int_X g_\star(\sigma)\wedge \varphi=
<g_\star(\sigma),\varphi>$ by formula \eqref{fundamental} of Proposition~\ref{pushforwardResults}
\end{proof}

As described in \cite{federer}, an inner product on a vector space 
$V$ can be viewed as an isomorphism $\ell\colon V\to V^*$
satisfying certain properties. The inverse of $\ell$ gives the induced
inner product on $V^*$. The fact that $<v,w>\leq \|v\|\cdot \|w\|$
with equality iff $v$ and $w$ are scalar multiples implies that
the inner product norm on $V^*$ is the same as the operator norm
of $V^*$ acting on $V$.

The induced map $\bigwedge^k\ell\colon \bigwedge^k V\to \bigwedge^k V^*$
gives an inner product on $\bigwedge^k V$. We call this the canonical
inner product on $\bigwedge^k V$ induced by the inner product on $V$.
Hence, given a Riemannian metric on $X$, there are canonical
smoothly varying inner products on $\bigwedge^k T_x X$ and 
$\bigwedge^k T_x^* X$ for each $x\in X$. At any point $x\in X$
we define $\| \bigwedge^k D_xf\|$ to be the operator norm of the linear
function $\bigwedge^k D_xf\colon \bigwedge^k T_x X\to \bigwedge^k T_{f(x)} X$.
We define $\| \bigwedge^k Df\|$ to be the $\bounded$ norm
of the map $x\mapsto \|\bigwedge^k D_x f\|$.
Also, given a $k$ form $\varphi$ we define 
the {\it comass} $\|\varphi\|_{\bounded}$ of $\varphi$ to be the $\bounded$ norm of the function
$x\mapsto \|\bigwedge^k \varphi_x \|$.
It is clear that the $k$ forms
with the comass norm is a Banach space.
We now show that the $k$ forms with $\bounded$ coefficients
are naturally lenient currents. 
We start by defining the action
on nimble forms.

\begin{definition}
\label{formsAsCurrent}
Given an $n-k$ form $\current{C}$ with $\bounded$ coefficients
we define
\[<\current{C},p_*(\sigma)>=(-1)^{\binom{n-k+1}{2}}\int_X \current{C}\wedge p_\star(\sigma)\]
\end{definition}

\begin{lemma}
The space $\forms{n-k}(\bounded)$ of $n-k$ forms with $\bounded$ coefficients
under the comass norm includes continuously
into $\lenient_{k}(X)$ where the action of $\current{C}\in\forms{n-k}(\bounded)$
on some $\varphi=\sum_i f_{i*}(\sigma_i)\in\nimble^k(X)$, with each $f_i\in \cat_X$
and each $\sigma_i\in\forms{k}(\smooth)$ is given by
\[<\current{C},\varphi>\equiv \sum_i\int_X f_i^*(C)\wedge \sigma_i.\]
\end{lemma}
\begin{proof}
The assumption that $X$ is compact means that any two Riemannian
metrics on $X$ are comparable. Choose one so the notion of
the comass norm makes sense. 
The result is then a straightforward consequence of
equation~\eqref{fundamental}, Lemma~\ref{nimbleAreForms}, 
and our definitions.
\end{proof}

\begin{remark}
It follows that a current with local $\forms{k}(\bounded)$
potentials is also a lenient current. 
\end{remark}

\begin{remark}
\label{pullbacksAgree}
Given a member $\current{C}$ of $\forms{k}(\bounded)$
then $f^*(\current{C})$ is the same whether 
done as a lenient current or as a form.
This, along with the fact that $\ed f^*=f^*\ed$
justifies the ad hoc pullback of closed positive
$(1,1)$ currents used so successfully in holomorphic dynamics.
Similarly $\ed \current{C}$ gives the same result whether calculated
as a lenient current or a form if 
$\current{C}\in \forms{k}(C^1)$.
\end{remark}

\subsection{H\"older Lemmas}

We will want to apply Corollary~\ref{regularity} to show that each
eigencurrent we construct has local $\ed$ potentials (or $\ed\ed^c$
potentials in the holomorphic case) which are forms with H\"older
continuous coefficients. In order to do this we will need 
a few facts which we include here in order to avoid 
having to include regularization
results as afterthoughts to our main theorems.

\begin{observation}
\label{holderOne}
Let $\holder_\alpha$ be the functions with coefficients that are H\"older
of exponent at least equal to some fixed $\alpha > 0$. 
Since diffeomorphisms
preserve H\"older exponents and averages of H\"older functions
are H\"older then we take it as clear that
Corollary~\ref{subsheafCohomology} applies to
show that $H^1(X,\sheaf{A'})=H^1(X,\sheaf{A})$ where 
$\sheaf{A'}$ is the closed members of 
$\forms{k}(\holder_\alpha))$
and $\sheaf{A}$ is the closed degree $k$ currents.
\end{observation}

\begin{lemma}
\label{holderTwo}
Let $X$ be a compact manifold (real or complex) with a Riemannian metric
and of real dimension $n$.
Let $f\colon X\to X$ be a smooth map. Then local coordinate charts
$U_i$ can be chosen on $X$ (each representing a convex open subset of $\R^n$)
so that there is a positive constant
$1 < M$ so that for any $k$ form $\varphi$, there exist constants
$c,C > 0$ such that writing each $f^{k*}(\varphi)$ in any of the charts $U_i$ as
\[f^{k*}(\varphi)=\sum a_{ki} \ed x^{\wedge i}\]
then each function $a_{ki}$ satisfies
\begin{equation}
\label{firstFact}
\sup_{x\in U_i} \abs{a_{ki}} \leq c\cdot \comassNorm{f^{k*}(\varphi)}
\end{equation}
and for each $j\in 1,\dotsc,n$,
\[\sup_{x\in U_i} \Big\arrowvert\pder{a_{ki}}{{x_j}}\Big\arrowvert\leq C\cdot M^k.\]
\end{lemma}
\begin{proof}
Equation~\eqref{firstFact} is a basic fact.

The rest is a straightforward consequence of realizing a self map
of a manifold as being made up of a bunch of maps between
different coordinate patches in $\R^n$.
That is, one chooses an open cover of patches $U_i$ of $X$.
Each patch is realized in $\R^n$ as a round ball.
Thinking of each patch as lying in $\R^n$
then we can find explicit maps from between
open subsets of $\R^n$ of the form 
$p_{ij}\colon U_i\cap f^{-1}(U_j)\to U_j$.
By shrinking each open ball $U_i$ a small amount
the resulting patches still cover $X$ but the derivatives
of the maps $p_{ij}$ are all now bounded (since we are working
on relatively compact subsets of the previous maps $p_{ij}$).

Then given any $x$ we can keep track of which patch
$f^k(x)$ is in at each time and can then realize
the map $f^k(x)$ as a composition 
$p_{i_1i_2}\circ p_{i_2i_3}\circ\dotsb\circ p_{i_{k-1}i_k}$.
Since each partial derivative of each $p_{ij}$ is uniformly bounded
then any partial derivative of the composition grows at most
exponentially with $k$ and we are done.
\end{proof}

The following observation will also be useful:

\begin{lemma}
\label{holderThree}
If there are positive constants $c,C,m,M$ with 
$m < 1 < M$ such that
a sequence of smooth functions $h_k$ on an open convex set
$U\subset \R^n$ satisfies 
\[\supNorm{h_k} < c\cdot m^k\]
and 
\[\Big\Arrowvert \pder{h_k}{x_j}\Big\Arrowvert_{\text{sup}} < C\cdot M^k\]
for all $k\in 0,1,2,\dotsc$ then
$h_1+h_2+h_3+\dotsc$ converges to
a bounded continuous function 
which is H\"older of any exponent $\alpha < \frac{\log(m)}{\log(m/M)}$.
\end{lemma}
\begin{proof}
The proof is elementary.
\end{proof}

\subsection{Eigencurrents for Cohomologically Expanding Smooth Maps}

We will call a section $V$ of $\bigwedge^k TX$ a $k$-vector field.
We define $\|V\|_{\bounded}$ to be the $\bounded$ norm of the function
$x\mapsto \|V_x\|$. 
Whether Theorem~\ref{main}
applies to a map will depend the size of $\|\bigwedge^k Df\|$.
Replacing $f$ with an iterate does not 
affect the needed estimate so we make the following
definition.

\begin{definition}
We define $\maxMult_k$ to be the limit supremum as $j\to \infty$ of 
$\|\bigwedge^k D(f^j)\|^\frac{1}{j}$.
It follows that $\maxMult_1 \geq e^\lambda$ for any Lyapunov 
exponent $\lambda$ and that $\maxMult_k \leq \maxMult_1^k$ (\cite{federer} page 33).
\end{definition}

We let $\sheaf{B}$ be the sheaf $\forms{k-1}(\bounded)$.
The norm $\|\cdot\|_\infty$ clearly makes $\shGamma{B}$ into a Banach space.
Given a member $\current{B}\in\shGamma{B}$, since the operator
norm on each $\bigwedge^k T_xX$ is equal to the norm already
defined on $\bigwedge^k T_x^*X$ for each $x\in X$ then
$\|\current{B}\|_\infty$ is equal to supremum of the $\bounded$ norm of 
the function $x\mapsto \current{B}(V_x)$ as $V$ varies
over all $\bounded$ $k$-vector fields of norm no more than one.

\begin{theorem}
\label{boundedPotentials}
Given $f\colon X\to X$ an a map in $\cat_X$ for the 
compact orientable manifold $X$, assume that
$c\in \deRham^k(X)$ is a cohomology class (using either real or complex
deRham cohomology)
which is an eigenvector for $f^*$ with eigenvalue $\beta$.
Assume also that $\vert \beta\vert > \maxMult^{k-1}$.
Then there exists a unique eigencurrent $\current{C}$
with local $\forms{k-1}(\bounded)$ potentials
representing the class $c$.
Moreover $\current{C}$ has local $\forms{k-1}(\holder)$ potentials. 

Also, given any neighborhood $U\subset X$ of any point in
the support of $\current{C}$, then for every 
lenient current $\current{C'}$ with local $\forms{k-1}(\bounded)$ potentials
and which represents the cohomology class $c$ then
$f^k(U)\cap \supp \current{C'}\nempty$ for all large $k$.

Assume that the linear map $f^*\colon\deRham^k(X)\to\deRham^k(X)$
is dominated by a single simple real eigenvalue $r$. 
Given $\current{C}'$ any current which has local
$\forms{k-1}(\bounded)$ potentials and which represents a cohomology
class in the $\maxMult^{k-1}$
chronically expanding subspace of $\deRham^k(X)$,
then the successive rescaled pullbacks $f^{k*}(\current{C}')/r^k$
of $\current{C}'$ converge to a multiple of $\current{C}$ in the
sense of lenient currents (and thus also in the sense of currents).
\end{theorem}
\begin{proof}
We let $\sheaf{B}=\forms{k-1}(\bounded)$, $\sheaf{A}$ and $\sheaf{C}$ be the
kernel and image respectively of $\sheaf{B}\lto{\ed}\lenient^k$.
By Theorem~\ref{subsheafCohomology}, $H^1(X,\sheaf{A})$ can be
canonically identified with $H^k(X,\K)$. Since $\sheaf{B}$ is $\Gamma$-acyclic
then every member of $H^1(X,\sheaf{A})$ is a closed bundle with respect
to the short exact sequence $\sheaf{A}\to\sheaf{B}\to\sheaf{C}$.

From Lemma~\ref{pullbackBasics}
there is an induced
$f$ cohomorphism of the short exact sequence
$\sheaf{A}\lto{\iota}\sheaf{B}\lto{\ed}\sheaf{C}$.
Also $\shGamma{C}$ is a space of lenient currents by Lemma~\ref{pullbackBasics} 
and thus has a natural structure as a topological vector space.
If a sequence $\current{B}_i\in\shGamma{B}$ converges to
$\current{B}\in\shGamma{B}$ then
$<\ed \current{B}_i,\varphi>=\int_X B_i\wedge \ed \varphi=\int_X B\wedge \ed\varphi=<\ed\current{B},\varphi>$
so the map $\ed\colon \shGamma{B}\to\shGamma{C}$ is continuous.

The cohomomorphism $\mshGamma{f}{B}$ is pullback $f^*$
of differential forms. 
Fixing any real $\alpha$ satisfying $\maxMult_{k-1} < \alpha <\vert \beta\vert$
it is clear from the definition of $\maxMult_{k-1}$ 
that one can choose a real $d >0$ such that 
$\|\bigwedge^{k-1} D(f^\ell)\| \leq d\cdot \alpha^\ell$ for all $\ell\in \N$.
The $\ell^\text{th}$ pullback $f^{\ell*}(\current{B})$ of $\current{B}\in\shGamma{B}$
satisfies $\|f^{\ell*}(\current{B})\|_\infty=
\sup_{V}\|\current{B}(\bigwedge^k D(f^\ell)(V))\|_\infty$
where the supremum is taken over all $k$-covector fields $V$ with $\|V\|_\infty \leq 1$.
However $\bigwedge^k D(f^\ell)(V)$ is a $k$-covector field of norm no more than
$\|\bigwedge^k D(f^\ell)\|$, so 
$\|f^{\ell*}(\current{B})\|_\infty \leq \|\current{B}\|_\infty\cdot \bigwedge^k D(f^\ell)\|_\infty\leq d\cdot \alpha^\ell\|\current{B}\|_\infty$.

Given any $W$ in the $\maxMult_{k-1}$ chronically expanding subspace of 
$H^k(X,\K)$, we can alter our choice of $\alpha > \maxMult_{k-1}$
so that $W$ also lies in the $\alpha$ chronically expanding subspace of 
$H^k(X,\K)$.

We can therefore apply Theorem~\ref{main} to conclude that there 
is a (unique) map $\kappa\colon W\to \shGamma{C}$ such that
$f^*\kappa=\kappa f^*$, where the first $f^*$ is pullback
of currents and the second is pullback on $H^k(X,\K)$.

In fact $\kappa(W)$ lies in the space of currents with locally H\"older
potentials (meaning $\forms{k-1}(\holder)$ potentials) by applying
Corollary~\ref{regularity} in conjunction with Observation~\ref{holderOne},
Lemma~\ref{holderTwo} and Lemma~\ref{holderThree}.
The second half of the Theorem is a consequence of
equation~\eqref{iterationFormula}.
\end{proof}

\begin{remark}
Theorem~\ref{boundedPotentials} 
gives regular degree one eigencurrents for every eigenvalue of 
$f^*\colon H^1(X,\K)\to H^1(X,\K)$ of norm greater one without
requiring any constraints on the local behavior of $f$. The degree
one eigencurrents seem to be, in some sense, more robust than currents
of lower dimension, including invariant measures. Moreover since 
codimension one closed submanifolds are closed currents with local
$\forms{0}(\bounded)$ potentials then successive rescaled preimages
of such manifolds in the right cohomological class will converge
to the eigencurrent.
\end{remark}

\begin{remark}
The fact that eigencurrents constructed via Theorem~\ref{boundedPotentials}
 have local potentials which are forms does
not imply their support has positive Lebesgue measure 
as the classical example of a monotonic nonconstant function which
is constant on a set of full measure shows.
\end{remark}

\begin{remark}
\label{moreGenerality}
The assumption that $f^*\colon \deRham^1(X)\to\deRham^1(X)$
is dominated by a single simple real eigenvalue $r$ is
not essential, but just meant to handle the simplest case.
In fact the proof actually shows that if $W$ lies in the 
$\maxMult_{k-1}$ chronically expanding subspace of 
$H^k(X,\K)$ then every current in the invariant plane
$\kappa(W)\subset \shGamma{C}$ of currents
has local $\forms{k-1}(\holder)$ potentials
and any current with cohomological class in $W$ with
local $\forms{k-1}(\bounded)$ potentials is attracted
to $\kappa(W)$ under successive rescaled pullback.
\end{remark}

Since measures are of particular interest in dynamics, we note that 
$H^1(X,\forms{n-1}(\bounded))=H^n(X,\K)=\K$ by Corollary~\ref{subsheafCohomology}
so there is a unique $f^*$ eigenvalue and it is precisely the topological degree of $f$.
We thus obtain:

\begin{corollary}
Given that $\maxMult_{n-1} < \deg f$ then there is a unique
dimension zero eigencurrent $\current{C}$ with $\forms{k-1}(\bounded)$ potentials
(and in fact it has $\forms{k-1}(\holder)$ potentials)
and the successive rescaled preimages of any $\current{C}'$
with $\forms{k-1}(\bounded)$ potentials converge to $\current{C}$.
If additionally there is no point $x\in X$ about which $f$ is locally
an orientation reversing diffeomorphism 
then $\current{C}$ (and every other member of $\kappa(W)$) 
is a positive distribution and is therefore a Radon measure.
\end{corollary}
\begin{proof}
Since $f^*$ pulls back dimension zero currents (i.e. distributions) which are positive
to distributions which are positive then by
Corollary~\ref{positivity} the distribution $\current{C}$ is positive.
It is therefore a Radon measure (see e.g. \cite{hirsch-lacombe} page 270). 
\end{proof}

\begin{remark}
In the case where $f$ is orientation reversing on 
some parts of $X$ (but not on all of $X$) some special remarks apply.
If it happens that successive rescaled images of some point
converge to a dimension zero eigencurrent then since preimages of points are counted
with multiplicity then when pulled back through a portion of $X$
on which $f$ reverses orientation the sign of a point is flipped.
Thus in this case the eigencurrent may not describe so much the
distribution of preimages as the {\it relative density}
of preimages counted negatively as compared to those counted positively. 
The number of actual preimages of a point 
may grow exponentially faster than the degree
of the map in such cases so that dividing by
the degree does not yield a measure in the limit 
unless some such ``cancellation'' takes place in the limit. One would expect that 
the corresponding eigencurrents have local potentials which are not of bounded variation
in such a case.
\end{remark}

\subsection{Eigencurrents for Smooth Covering Maps}

We will call a covering map which is locally a diffeomorphism
a {\it smooth covering map}.
We now consider the special case of smooth self covering maps 
$f\colon X\to X$ of a compact smooth orientable manifold $X$.
We show that in this case we have a substantially broader
collection of currents whose successive pullbacks
converge to an eigencurrent, albeit we need different
estimates for Theorem~\ref{main} to apply. We will pull back currents
by pushing forward forms with $f_\star$. Since the regular set
of $f$ is all of $X$ then $f_\star$ is a well defined operator
from smooth forms to smooth forms.

\begin{definition}
For a map satisfying the hypothesis of Proposition~\ref{pushforwardResults}
we define the operation $f^*$ from currents on $X$ to currents on $Y$
by 
\[<f^*(\current{C}),\alpha>\equiv <\current{C},f_\star(\alpha)>.\]
Clearly $f^*$ preserves the dimension of a current.
\end{definition}

Let $\sheaf{M}_{k-1}$ be the sheaf 
for which
$\sheaf{M}_{k-1}(U)$ is the Banach space of bounded linear operations 
on the topological vector space comprised of $\forms{k-1}(\smooth)(U)$
with the $\|\cdot\|_\infty$ norm. Equivalently, $\sheaf{M}_{k-1}$
is the sheaf of dimension $k-1$ currents of finite mass.

Choose a Riemannian metric on $X$.
If $f\colon X\to X$ is a smooth cover then for each $x\in X$
and each $\ell\in \N$, $D_x(f^\ell)\colon T_x X\to T_{f^\ell(x)}X$
is invertible. We let $\nu_k(x,\ell)$ be the operator norm of 
the inverse of $\bigwedge^k D_x(f^\ell)\colon \bigwedge^k T_xX\to\bigwedge^k T_{f^\ell(x)} X$.
We define $\nu_k(\ell)=\sup_{x\in X} \nu_k(x,\ell)^{1/\ell}$.
We define $\nu_k=\limsup_{\ell\to\infty}\nu_k(\ell)$.
The iterated pushforward operation $f^\ell_\star\colon\forms{k-1}(\smooth)(X)\to\forms{k-1}(\smooth)(X)$
satisfies $\|f^\ell_\star(\varphi)\|_\infty\leq \nu_k(\ell)\cdot \|\varphi\|_\infty$
as is straightforward to verify. If $f$ is invertible
then $\nu_k$ is a bound on the growth of the
$k^\text{th}$ wedge product of the derivative under
$f^{-1}$. For non-invertible $f$, $\nu_k$ represents a bound
on the growth of the $k^\text{th}$ wedge product of the derivative 
under any sequence of successive branches of $f^{-1}$.

\begin{theorem}
\label{boundedOperatorPotentials}
Given $f\colon X\to X$ a smooth self covering map and that
$c\in \deRham^k(X)$ is a cohomology class (using either real or complex
deRham cohomology)
which is an eigenvector for $f^*$ with eigenvalue $\beta$.
Assume also that $\vert \beta\vert > \nu_{k-1}$.
Then there exists a unique eigencurrent $\current{C}$
with local $\sheaf{M}_{k-1}$ potentials
representing the class $c$.
Moreover $\current{C}$ has local $\forms{k-1}(\continuous)$ potentials. 
Consequently $\current{C}$ is a current of order one.

Also, given any neighborhood $U\subset X$ of any point in
the support of $\current{C}$, then for every 
lenient current $\current{C'}$ with local $\sheaf{M}_{k-1}$ potentials
and which represents the cohomology class $c$ then
$f^k(U)\cap \supp \current{C'}\nempty$ for all large $k$.

Assume that the linear map $f^*\colon\deRham^k(X)\to\deRham^k(X)$
is dominated by a single simple real eigenvalue $r$. 
Given $\current{C}'$ any current which has local
$\sheaf{M}_{k-1}$ potentials and which represents a cohomology
class in the $\nu^{k-1}$
chronically expanding subspace of $\deRham^k(X)$,
then the successive rescaled pullbacks $f^{k*}(\current{C}')/r^k$
of $\current{C}'$ converge a multiple of $\current{C}$.
\end{theorem}
\begin{proof}
We let $\sheaf{A}$ and $\sheaf{C}$ be the kernel
and image respectively of $\ed \colon \sheaf{M}_{k-1}\to \curDeg{k}$.
Since $\ed f_\star=f_\star \ed$ then pullback of currents
gives an $f$ cohomomorphism of the short exact sequence of sheaves
$\sheaf{A}\to\sheaf{M}_{k-1}\to\sheaf{C}$.

Since $\Gamma\sheaf{M}_{k-1}$ is the continuous linear operators on
a normed vector space then it is a Banach space. From the observations
previous to the statement of Theorem~\ref{boundedOperatorPotentials}
one concludes that for any $\alpha > \nu_{k-1}$ there is a constant
$d > 0$ such that $\|f^{\ell *}(\current{B})\|\leq d\cdot \alpha^k \|\current{B}\|$
for all $\ell \in \N$.

Since $\shGamma{C}$ is a space of currents it is naturally a topological
vector space over $\K$.

The map $f^*\colon\shGamma{C}\to\shGamma{C}$ is continuous since 
if $\current{C}_i\to \current{C}$ in $\shGamma{C}$ then 
$<f^*(\current{C}_i),\varphi>=<\current{C}_i,f_\star(\varphi)>\to
<\current{C},f_\star(\varphi)>=<f^*(\current{C}),\varphi>$.

If $\current{P}_i\to\current{P}$ in $\Gamma\sheaf{M}_{k-1}$
(using the Banach space structure)
then $\|\current{P}_i-\current{P}\|\to 0$ by assumption
then $\|\current{P}(\ed \varphi)-\current{P}_i(\ed \varphi)\|\leq
\|\current{P}-\current{P}_i\|\cdot \|\ed \varphi\|\to 0$.
Hence $<\ed \current{P}_i,\varphi>=\current{P}_i(\ed\varphi)\to\current{P}(\ed \varphi)
=<\ed \current{P},\varphi>$ and so we conclude that 
the map $\ed\colon \Gamma\sheaf{M}_{k-1}\to\shGamma{C}$ is continuous. 

Given any $W$ in the $\nu_{k-1}$ chronically expanding subspace of 
$H^k(X,\K)$, we can alter our choice of $\alpha > \nu_{k-1}$
so that $W$ also lies in the $\alpha$ chronically expanding subspace of 
$H^k(X,\K)$.

We can therefore apply Theorem~\ref{main} to conclude that there 
is a (unique) map $\kappa\colon W\to \shGamma{C}$ such that
$f^*\kappa=\kappa f^*$, where the first $f^*$ is pullback
of currents and the second is pullback on $H^k(X,\K)$.

In fact $\kappa(W)$ in the currents with locally continuous
potentials by applying
applying Corollary~\ref{regularity} in conjunction with Observation~\ref{holderOne},
Lemma~\ref{holderTwo} and Lemma~\ref{holderThree}.
The second half of the Theorem is a consequence of
equation~\eqref{iterationFormula}.
\end{proof}

\begin{proposition}
\label{preimagesOfSubmanifolds}
Let $Y$ be an oriented codimension $k$ submanifold of $X$.
If the cohomological class of $Y$ (as a current) lies in
the $\nu_{k-1}$ chronically expanding subspace of $H^k(X,\K)$
then the successive rescaled preimages of $Y$ converge to
the invariant plane of currents $\kappa(W)$. 
If $f^*\colon H^k(X,\K)\to H^k(X,\K)$ is dominated by a single
real eigenvalue $r > \nu_{k-1}$ then the successive rescaled preimages of $Y$
converge to a multiple (possibly zero) of the $r$ eigencurrent.
In particular, if $\nu_{n-1} < \deg f$ then the successive
rescaled preimages of any point converge to the unique
invariant measure with $\sheaf{M}_{n-1}$ potentials.
\end{proposition}
\begin{proof}
This follows immediately from Theorem~\ref{boundedOperatorPotentials}
if we show that $Y$ has local potentials in $\sheaf{M}_{k-1}$. 
This is equivalent to showing that locally $Y=\ed P$
where $<P,\varphi> \leq a\cdot \|\varphi\|_\infty$ for some $a > 0$.
Let $B$ be a ball in $\R^n$ and $Y_0$ a $k$-plane in $\R^n$.
Then there is a $k+1$ half plane $P$ such that, as currents in 
$U$, $\partial P=Y_0$. Moreover it is clear that
$<P,\varphi>\leq a\|\varphi\|_\infty$ for some real $a > 0$.
(There are also local potentials for $Y$ which are given by forms
with $L^1_\text{loc}$ coefficients. These can be constructed by choosing a projection
$\pi$ from $U\setminus Y_0$ to a codimension one cylinder $C$ with axis $Y_0$,
and choosing a volume form $\sigma$ on $C$. The local potential is the
pullback $\pi^*(\sigma)$.)
\end{proof}

\begin{remark}
As with Theorem~\ref{boundedPotentials}, Theorem~\ref{boundedOperatorPotentials}
gives regular degree one eigencurrents for every eigenvalue of 
$f^*\colon H^1(X,\K)\to H^1(X,\K)$ of norm greater one without
requiring any constraints on the local behavior of $f$. 
In holomorphic dynamics much progress has been made in constructing
degree one eigencurrents and then constructing dynamically important invariant measures
via a generalized wedge product (see the references cited at the beginning
of Section~\ref{holEndSection}). 
\end{remark}

\begin{remark}
The proof of Proposition~\ref{preimagesOfSubmanifolds} 
could clearly be modified to apply to many singular manifolds as well.
\end{remark}

\section{Holomorphic Endomorphisms}
\label{holEndSection}

We now restrict our interest to holomorphic dynamics.
Thus all manifolds are assumed to be complex manifolds and 
all maps are assumed to be holomorphic unless stated otherwise.

Holomorphic endomorphisms of the Riemann sphere have been
studied in great detail. For endomorphisms much 
of the theory is still in its beginnings. Much attention
has been paid to holomorphic automorphisms of $\C^2$ \cite{friedland-milnor},
\cite{forn:henon},
\cite{hubbard_oberste-vorth:henon1}, \cite{hubbard_oberste-vorth:henon2},
\cite{bedfordsmillie1}, \cite{bedfordsmillie2}, \cite{bedfordsmillie3}, 
\cite{bedfordlyubichsmillie4}, \cite{bedfordsmillie5}, \cite{bedfordsmillie6}, 
\cite{bedfordsmillie7} 
or K3 surfaces \cite{cantat}, \cite{mcmullen:k3},
the major developments for endomorphisms have been
on $\Proj^n$, \cite{forn:cdhd}, \cite{forn:cdhd1},
\cite{forn:cdhd2}, \cite{forn:examples}, \cite{forn:classification},
\cite{jonsson:example}, \cite{jonsson-favre}, 
\cite{ueda:fatou}, \cite{ueda:normfam}, \cite{ueda:fatmaps}.
Recent significant developments have been made for endomorphisms
of Kahler manifolds in \cite{dinhsibony-greenCurrents}.
The paper \cite{dinhsibony-greenCurrents} shows existence of eigencurrents
(or Green's currents) for endomorphisms of Kahler manifolds under a simple
condition on the comparative rates of growth of volume
in two different dimensions. They also show that a specific weighted 
sum of an arbitrary closed positive smooth current
will converge to the Green's current, and that the Green's
current has a H\"older continuous potential.
In this setting our theorem shows that arbitrary
(rescaled) preimages of a broader class of currents will
converge to the Green's current. 
A wide variety of results have been proven in these various circumstances
either showing the existence of invariant currents, showing convergence
of currents to invariant currents, or studying the properties
of these invariant currents. We include here results that follow
from the method of this paper, which we are sure substantially overlap
with existing results. Presumably our cohomologicaly lifting 
theorem could be used in conjuction with Theorem~\ref{main}
to show existence of higher degree
$(k,k)$ currents given certain bounds on local growth rates.

\subsection{$\ed\ed^c$ Cohomology}

Let $Z$ be a complex manifold
and let $f\colon Z\to Z$ be a holomorphic self map of $Z$.
Let $\pluriharmonic$ be the sheaf of pluriharmonic functions,
let $\bounded$ be the sheaf of locally bounded functions, and 
let $\sheaf{C}$ be the sheaf of currents with local potentials
in $\bounded$, i.e. currents locally of the form $\ed\ed^cb$, for
$b$ a locally bounded function. The members of $\sheaf{C}$
are closed $(1,1)$ currents on $Z$.

Using the usual pullback on functions, and the induced pullback
on currents with function potentials (i.e. pullback the current by pulling back its local
potentials), then we get a self cohomomorphism of the exact sequence
of sheaves
\begin{equation}
\label{holSES}
\pluriharmonic\to \bounded\lto{\ed\ed^c}\sheaf{C}.
\end{equation}

We note that $H^1(Z,\pluriharmonic)$ is a finite dimensional
$\R$ vector space
as can be seen from the long exact sequence for the short
exact sequence $\R\to \hol\to \pluriharmonic$ where the
first map is inclusion and the second takes the imaginary part.
The terms $H^1(Z,\hol)\to H^1(Z,\pluriharmonic)\to H^2(Z,\R)$
give the finite dimensionality since $\hol$ is a coherent analytic
sheaf (see e.g. \cite{taylor} page 302).

Then from Theorem~\ref{main} we obtain:

\begin{corollary}
\label{cor-closedOneOne}
Given $v$ any closed eigenbundle of $\hbundles{Z}$
for $f^*$ with eigenvalue $r > 1$, 
there is a unique closed $(1,1)$ current $\current{C}$ 
such that 
$\lim_{k\to\infty} f^{k*}(\current{C}')/r^k$ converges
to $\current{C}$ for any divisor $\current{C}'$ of $v$.
\end{corollary}

\begin{remark}
We note that the terms ``closed eigenbundle'' and ``divisor'' in Corollary~\ref{cor-closedOneOne}
are understood using the long exact sequence for~\eqref{holSES}.
\end{remark}

We can apply Corollary~\ref{regularity}
to show that 
\begin{corollary}
Any such invariant current $\current{C}$ so obtained
has H\"older continuous local potentials.
\end{corollary}
\begin{proof}
The result follows from Lemma~\ref{holderOne}, Lemma~\ref{holderTwo}, the fact
that the $\ed\ed^c$ closed H\"older continuous functions are the same
as the $\ed\ed^c$ closed $\bounded$ functions and
from Corollary~\ref{regularity}.
\end{proof}

Also from Observation~\ref{positivity},
\begin{corollary}
If $v$ has a plurisubharmonic
section the current $\current{C}$ is positive.
\end{corollary}

\section{Result via Invariant Sections}

We stated early on that our construction of invariant members
of $H^0(\sheaf{C})$ for a self cohomomorphism
of a short exact sequence $\sheaf{A}\to\sheaf{B}\to\sheaf{C}$
of sheaves could be done in terms of finding invariant sections
of bundles. We illustrate this here in a specific case where
we can take advantage of geometry to make further conclusions.
Finding an invariant section of a bundle is equivalent to finding
an invariant trivialization of the bundle, and we will
make our initial statement in terms of a trivialization.

Let $Z$ be a compact complex manifold. Let $f\colon Z\to Z$
be a holomorphic endomorphism. Let $p\in\hbundles{Z}$
be an eigenvector for $f^*$ with real eigenvalue $\lambda$
of norm greater than one. If $f^*$ were to have complex eigenvalues
of interest, an analogous construction can be made
to the one that follows.

We note that there is a canonical bundle map
$\tilde{f}\colon f^*(p)\to p$ which gives the
map $f$ on the base space. It is easy to show that
there is a map $\sigma\colon p\to \lambda p$
which is the identity on the base space
and takes the form $r\mapsto \lambda r + b$
on the fibers, where $b$ is a constant. What is
more, the map $\tau_\lambda$ is easily seen to be unique
up to the addition of a constant.
Then define the map $\lift{f}\colon p\to p$ to be the composition
of 
\[ p\overset{\tau_\lambda}{\to} \lambda p=f^*(p)\overset{\tilde{f}}\to p.\]
Then $\lift{f}$ is the map $f$ on the base space and takes the form
$r\mapsto \lambda r + b$ on the fibers.

Since every pluriharmonic bundle is trivial as a smooth 
bundle,
then we can choose a smooth trivialization 
$t\colon p\to \R$, i.e. $t(a+r)=\sigma(a)+r$ for any
$a\in p$, $r\in R$, where $a+r$ is computed in the fiber containing $a$.

\begin{theorem}
There is a unique continuous trivialization
$\greens\colon p\to \R$ such that:
\[\greens(a+r)=\greens(a)+r\text{\quad for $a\in p$ and $r\in \R$},\]
\[\greens(\lift{f}(a))=\lambda\cdot \greens(a)\text{\quad for $a\in p$},\]
moreover
\[\greens=\lim_{k\to\infty} \lambda^{-k}\circ t\circ \lift{f}\cp{k}\]
and the limit converges uniformly. Finally, the zero set
of $\greens$ is the image of a section $g\colon Z\to p$
and is exactly the set of points whose forward image under $\lift{f}$
remains bounded.
\end{theorem}
\begin{proof}
Define a function $T\colon p\to\R$ by
\[T(a)\equiv t\bigl(\lift{f}(a)\bigr)-\lambda\cdot t(a).\]
Note that $T$ descends to a well defined continuous
function $T\colon Z\to \R$ since for an arbitrary $r\in \R$
one has $T(a+r)=t\bigl(\lift{f}(a+r)\bigr)-\lambda\cdot t(a+r)
=t(\lift{f}(a)+\lambda r)-\lambda\cdot(t(a)+r)=T(a)$.

One notes that since the function $T$ is necessarily bounded if $Z$
is compact then defining
\[\greens(a)\equiv t(a)+\lambda^{-1} \cdot T(a) + \lambda^{-2}T(\lift{f}(a)) +
 \lambda^{-3}T(\lift{f}\cp{2}(a))+\dotsb\]
gives a continuous function $\greens\colon p\to \R$ satisfying
the above two properties.

Assume $\greens_1$ and $\greens_2$ are two such functions.
Then $\Delta\equiv \greens_1-\greens_2\colon p\to \R$
is a function satisfying $\Delta(a+r)=\Delta(a)$ for $a\in p$ and $r\in \R$
so $\Delta$ descends to a continuous function $\Delta\colon p\to \R$
satisfying $\Delta(\lift{f}(a))=\lambda\cdot \Delta(a)$. However 
since $\lambda > 1$ one concludes that this is only possible if $\Delta\equiv 0$ since 
$M$ is compact so $\Delta(M)$ has compact image in $\R$.

It is easy to check using the definition
of $T$ that $\lambda^{-k}\circ t\circ\lift{f}\cp{k}(a)$
is exactly a partial sum of the first $k$ terms of
the above series and this gives the
convergence result. The conclusion about the section $g$ is trivial.
\end{proof}

The above construction can be carried through almost without modification
for any subspace of $\hbundles{Z}$ on which $f^*$ is expanding.
This gives an alternate way of understanding the convergence of preimages 
of sections. The point is that if $s$ is any section of $p$,
i.e. the potential of a current $C$, then $\frac{1}{\lambda}f^*(C)$
is a current with potential which is the setwise preimage
of $s$ under $\lift{f}$ (this is easy to confirm from the 
construction of $f$). The Green's trivialization $\greens$ shows
that $\lift{f}$ is uniformly repelling away from the image of 
the invariant section $g$. Thus as long as $s$ is bounded in $p$, (not 
even necessarily continuous), then the successive preimages 
of $s$ will converge uniformly to the section $g$. Since
uniform convergence of potentials implies convergence of currents
then the rescaled pullbacks of a current $C$ converge to 
the current with potential $g$. We already have this as a theorem,
so we have not restated it as such here. This is just an
alternative approach. Note that in the case where
$Z=\Proj^2$ \cite{jonsson-favre} has given far more precise control
of when the successive rescaled preimages of a current
will converge to the eigencurrent.

\subsection{Sections version with an Invariant Ample Bundle}



It is also interesting to consider the special case
where there is an invariant ample bundle with eigenvalue
$\lambda \geq 2$ an integer. Without loss of generality
we assume $\ell$ is very ample. 
The morphism of sheaves $\logabs{\cdot}\colon \hol^*\to \pluriharmonic$
induces a map from holomorphic line bundles to pluriharmonic bundles.
We let $p=\logabs{\ell}$ be the corresponding pluriharmonic bundle.

It is easy to see that there is a holomorphic map $\ell\to \ell^\lambda$
which is of the form 
$\sigma_\lambda\colon z\mapsto a z^\lambda$, $a\in \C^*$ on each fiber
and is the identity on the base space.
There is also a canonical
holomorphic map $\tilde{f}\colon f^*(\ell)\to \ell$ which
is a line bundle map and is $f$ on the base space.

One then defines the holomorphic map
$\hollift{f}\colon \ell\to\ell$ which is the composition of
\[\ell \overset{\sigma_k}{\to} \ell^k=f^*(\ell) \overset{\tilde{f}}{\to}\ell.\]
This map is of the form $z\mapsto a z^k$ on each fiber 
and is equal to the map $f\colon Z\to Z$ on the base space.
Let $\ell^*$ denote $\ell$ with its zero section removed,
so that $\logabs{\cdot}\colon \ell\to p$ is a well defined continuous map.
Since the preimage of the zero section of $\ell$ under $\hollift{f}$ is the zero
section then $\hollift{f}$ is a holomorphic self map of $\ell^*$.
It is easy to confirm that $\hollift{f}\colon \ell\to\ell$ can be rescaled so that the diagram
\[
\xymatrix{
{\ell} \ar[r]_{\hollift{f}} \ar[d]^{\logabs{\cdot}} & {\ell} \ar[d]^{\logabs{\cdot}} \\
{p} \ar[r]_{\lift{f}} & {p} \\ 
}
\]
commutes.


Our Greens trivialization $\greens\colon p\to \R$ can be pulled
back to give a Green's function $\G\colon \ell^*\to\R$ 
on the punctured bundle $\ell^*$. It satisfies 
$\G(\tilde{f}(w))=\lambda\cdot \G(w)$ and $\G(\beta w)=\G(w)+\logabs{\beta}$
for $w\in \ell$ and $\beta\in\C^*$. Since 
$\greens$ is a trivialization of an $\R$ bundle
over a compact space, $\greens$ is proper. Since
$\logabs{\cdot}\colon \ell^*\to p$ is proper then $\G$ is proper.
Thus, in this setting one can construct a Greens function that is exactly
analogous to the Green's function constructed on
$\C^{n+1}$ for a holomorphic endomorphism
of $\Proj^n$.  Potentially one could take advantage of 
the special geometry of very ample bundles to get
information about the dynamics in this situation.

\section{Bibliography}

\bibliographystyle{alpha}
\bibliography{refer}


\end{document}